\documentclass[a4paper,12pt]{amsart}
\usepackage{amsmath}
\usepackage{amsfonts}
\usepackage{amssymb}
\usepackage{hyperref}
\usepackage{fourier}
\usepackage[a4paper,left=2.5cm,right=2.5cm,top=2.5cm, bottom=2.5cm]{geometry}
\usepackage{lineno}
\usepackage[ddmmyyyy]{datetime}
\usepackage{tikz}

\newtheorem{lemma}{Lemma}
\newtheorem{theorem}[lemma]{Theorem}
\newtheorem{corollary}[lemma]{Corollary}
\newtheorem{proposition}[lemma]{Proposition}

\newcounter{rotcount}
\newtheorem{rot}{}[rotcount]

\newenvironment{proofof}{\noindent}{\hfill$\Box$\medskip}
\newenvironment{rotproof}{\noindent}{\hfill$\lozenge$\smallskip}

\newcommand{\ncont}{\nsubseteq}\newcommand{\cont}{\subseteq}

\newcommand{\del}{\backslash}

\newcommand{\defin}{\textbf}

\renewcommand{\int }{{\rm int}}

\title[Rooted prism-minors and disjoint cycles containing a specified edge]{Rooted prism-minors and disjoint cycles containing a specified edge}

\author[J.P. Costalonga]{J.P. Costalonga$^1$}
\author[T.J. Reid]{T.J. Reid$^2$}
\author[H. Wu]{H. Wu$^3$}

\thanks{$^1$Departamento de Matem\'atica, Universidade Federal do Esp\'irito Santo. Av. Fernando Ferrari, 514; Campus de Goiabeiras,
29075-910, Vit\'oria, ES, Brazil. {\upshape joaocostalonga@gmail.com}\\
$^{2,3}$Department of Mathematics. The University of Mississippi. University, MS 38677, USA. {\upshape $^2$mmreid@olemiss.edu}. {\upshape $^3$hwu@olemiss.edu}
}

\begin{document}
\begin{abstract}
 Dirac and Lov\'{a}sz independently characterized the $3$-connected graphs with no pair of vertex-disjoint cycles. Equivalently, they characterized all $3$-connected graphs with no prism-minors. In this paper, we completely characterize the $3$-connected graphs with an edge that is contained in the union of no pair of vertex-disjoint cycles. As applications, we answer the analogous questions for edge-disjoint cycles and for $4$-connected graphs and we completely characterize the $3$-connected graphs with no prism-minor using a specified edge.
\end{abstract}
\maketitle

\section{Introduction}

Results about vertex-disjoint and edge-disjoint cycles have received extensive attention by researchers in graph theory. There are two main lines of research in this area. One line of research provides results that give a sufficient condition for a graph to contain a certain number of vertex-disjoint or edge-disjoint cycles (see, for example, \cite{ Chiba, Corradi, Czygrinow2, Dirac-Erdos, Enomoto, Erdos-Posa, Kierstead1, Kierstead2, Kierstead4, Wang}). Another line of research provides results that classify graphs with no pairs of vertex-disjoint cycles (see, for example, \cite{Dirac} and \cite{Lovaz}). We give some results of the latter type here. Such results are particularly useful in the study of graph structure.

All graphs in this paper contain no loops nor parallel edges unless said otherwise. We use $K_5^-$ to denote the graph obtained by deleting an edge from the complete graph on five vertices. The graphs $K_{3,n}^{\prime}$, $K_{3,n}^{\prime \prime}$, and $K_{3,n}^{\prime \prime \prime}$ are obtained by adding, respectively, one, two, or three edges to a partite class of size three in the graph $K_{3,n}$. We denote by $W_n$ the wheel with $n$ spokes. The following result is independently due to Dirac~\cite{Dirac} and Lov\'{a}sz \cite{Lovaz}: 

\begin{theorem}\label{diracandlovasz}
A $3$-connected graph has no pair of vertex-disjoint cycles if and only if it is isomorphic to $W_n$, $K_5$, $K_5^{{-}}$, $K_{3,n}$, $K_{3,n}^{\prime}$, $K_{3,n}^{\prime \prime}$, or $K_{3,n}^{\prime \prime \prime}$ for some integer $n$ exceeding two.
\end{theorem}

We say that a graph $H$ is a \defin{minor} of a graph $G$ if $H$ is obtained from $G$ by contracting edges, deleting edges and deleting vertices. An \defin{$H$-minor} of a graph $G$ is a minor of $G$ that is isomorphic to a graph $H$.  The \defin{prism} is the graph obtained from two disjoint triangles by adding a perfect matching connecting the vertices of the different triangles. It follows from Menger's Theorem that Theorem \ref{diracandlovasz} is equivalent to the following result:

\begin{theorem}\label{dirac-prism}
A $3$-connected graph has a prism-minor if and only if it is not isomorphic to $W_n$, $K_5$, $K_5^{{-}}$, $K_{3,n}$, $K_{3,n}^{\prime}$, $K_{3,n}^{\prime \prime}$, or $K_{3,n}^{\prime \prime \prime}$ for some integer $n$ exceeding two.
\end{theorem}

We generalize both Theorems \ref{diracandlovasz} and \ref{dirac-prism}. First we discuss the generalization of Theorem \ref{dirac-prism}. We say that a minor $H$ of a graph $G$ \defin{uses} an edge $uv\in E(G)$ if $H$ has an edge $xy$ such that each $z\in\{x,y\}$ either is in $\{u,v\}$ or is obtained by the identification of a set of vertices intersecting $\{u,v\}$ when making the contractions to obtain $H$. Some authors call this a minor {\bf rooted} on $uv$. Determining when a class of graphs (resp. matroids) has a minor using a specified edges (resp. elements) is often useful and important in the study of structure of graphs and matroids. For instance, Seymour~\cite{Seymour-Adjacency} established a result on $K_4$ minors rooted on pairs of edges and derived results of disjoint paths on graphs. Results on the existence of $U_{2,4}$-minors rooted on one or two elements on matroids with $U_{2,4}$-minors (Bixby~\cite{Bixby} and Seymour~\cite{Seymour-2round}, respectivelly) are classic results in matroid theory; a property like this is called {\bf roundedness} and has some variations, for instance, $3$-connected graphs graphs with a minor isomorphic to $K_4$, $K_5\del e$ or $K_5$ have a minor rooted on the edge set of each edge-set of a triangle, provided they have these graphs as minors (\cite{Reid} and \cite{Czhou}) and $3$-conneted graphs with $K_{3,3}$-minors containing a degree $3$-vertex have a minor rooted on the edges adjacent to this vertex~\cite{Truemper}. Our first main result is the following generalization of Theorem \ref{dirac-prism}:

\begin{theorem}\label{thm-prism}
Suppose that $G$ is a $3$-connected graph on at least six vertices and $uv$ is an edge of $G$. Then $G$ has no prism-minor using $uv$ if and only if 
\begin{enumerate}
 \item [(a)] for some $n\ge 3$, $G\cong W_n$, $K_{3,n}$, $K'_{3,n}$, $K''_{3,n}$ or $K'''_{3,n}$ or
 \item [(b)] $G$ has a vertex $w$ such that $\{u,v,w\}$ is a vertex-cut of $G$ and each connected component $K$ of $G\del \{u,v,w\}$ is a tree with an unique neighbor of $w$ in respect to $G$.
\end{enumerate}
\end{theorem}

There are variants of Theorem \ref{diracandlovasz}, including \cite{Slilaty}. Motivated by Theorem \ref{diracandlovasz}, we consider a much larger class of graphs. The graphs in this class may contain vertex-disjoint cycles, but they do not contain vertex-disjoint cycles which union contains a specified edge. Equivalently, those are the graph with an edge $e$ with the property that, for each cycle $C$ containing $e$, $G-V(C)$ is a forest. The full and detailed characterization is made in Theorem \ref{main-cor}, the second main result in this paper. The statement of this theorem though requires a considerable amount of new terminologies and for this reason we will only state it in the end of Section \ref{sec-rope bridges}. Although the elaborated statement, Theorem \ref{main-cor} has pratical applications and is of independent interest. Indeed, it is used to prove our first main result, Theorem \ref{thm-prism}. Moreover, we apply it to  Theorems \ref{main-strong} and \ref{main-strong-2con} to completely characterize $3$-connected and $2$-connected graphs with no pair of edge-disjoint cycles which union contains a specified edge. A more succinct description of the graphs with no pair of vertex-disjoint cycles which union contains a specified edge is given in Theorem \ref{main}.

The following results completely characterizes the $3$-connected and $2$-connected graphs with no pair of edge-disjoint cycles which union contains a specified edge.

\begin{theorem}\label{main-strong}
If $G$ is a $3$-connected graph with an edge $e=u_1u_2$, then $G$ contains no edge-disjoint cycles using $e$ if and only if $G$ has internally disjoint $(u_1,u_2)$-paths $\alpha=u_1,v_1,\dots,v_n,u_2$ and $\beta=u_1,w_1,\dots,w_n,u_2$ not containing $e$ and there is a family $\mathcal{P}$ of pairwise disjoint pairs of consecutive elements of $\{1,\dots,n\}$ such that 
\begin{enumerate}
\item [(a)] $V(G)= V(\alpha)\cup V(\beta)$ and
\item [(b)] $E(G)=\{v_iw_j,v_jw_i:\{i,j\}\in \mathcal P\}\cup\{v_kw_k:k$ is in no member of $\mathcal P\}\cup E(\alpha)\cup E(\beta)\cup \{e\}$.
\end{enumerate}
\end{theorem}

\begin{figure}[h]
\centering
\begin{tikzpicture}
\tikzstyle{node_style} =[shape = circle,fill = black,minimum size = 2pt,inner sep=1pt]
\node[node_style] (u1) at (0.0,1.5) {};
\node[node_style] (u2) at (10,1.5){};
\node[node_style] (x1) at (1,1) {};
\node[node_style] (x2) at (2,1) {};
\node[node_style] (x3) at (3,1) {};
\node[node_style] (x4) at (4,1) {};
\node[node_style] (x5) at (5,1) {};
\node[node_style] (x6) at (6,1) {};
\node[node_style] (x7) at (7,1) {};
\node[node_style] (x8) at (8,1) {};
\node[node_style] (x9) at (9,1) {};
\node[node_style] (y1) at (1,2) {};
\node[node_style] (y2) at (2,2) {};
\node[node_style] (y3) at (3,2) {};
\node[node_style] (y4) at (4,2) {};
\node[node_style] (y5) at (5,2) {};
\node[node_style] (y6) at (6,2) {};
\node[node_style] (y7) at (7,2) {};
\node[node_style] (y8) at (8,2) {};
\node[node_style] (y9) at (9,2) {};
\draw (u1)--(x1)--(x2)--(x3)--(x4)--(x5)--(x6)--(x7)--(x8)--(x9)--(u2);
\draw (u1)--(y1)--(y2)--(y3)--(y4)--(y5)--(y6)--(y7)--(y8)--(y9)--(u2);
\draw (u1)--(0,0)--(10,0)--(u2);
\node () at (1,0.7) {$v_1$};
\node () at (2,0.7) {$v_2$};
\node () at (3,0.7) {$v_3$};
\node () at (4,0.7) {$v_4$};
\node () at (5,0.7) {$v_5$};
\node () at (6,0.7) {$v_6$};
\node () at (7,0.7) {$v_7$};
\node () at (8,0.7) {$v_8$};
\node () at (9,0.7) {$v_9$};
\node () at (1,2.3) {$w_1$};
\node () at (2,2.3) {$w_2$};
\node () at (3,2.3) {$w_3$};
\node () at (4,2.3) {$w_4$};
\node () at (5,2.3) {$w_5$};
\node () at (6,2.3) {$w_6$};
\node () at (7,2.3) {$w_7$};
\node () at (8,2.3) {$w_8$};
\node () at (9,2.3) {$w_9$};
\node () at (5,-0.2) {$e$};
\node () at (-0.1,1.7) {$u_1$};
\node () at (10,1.7) {$u_2$};
\draw (x1)--(y1);
\draw (x2)--(y3); \draw (x3)--(y2);
\draw (x4)--(y4);\draw(x5)--(y5);
\draw (x6)--(y7); \draw (x7)--(y6);
\draw (x8)--(y9); \draw (x9)--(y8);
\end{tikzpicture}
\caption{An example for Theorem \ref{main-strong} with $n=9$ and $\mathcal{P}=\big\{\{2,3\}, \{6,7\}, \{8,9\} \big\}$.}
\end{figure}
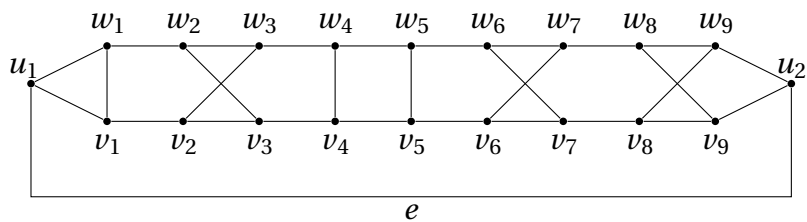

The following theorem generalizes Theorem \ref{main-strong} for $2$-connected graphs.

\begin{theorem}\label{main-strong-2con}
Suppose that $G$ is a $2$-connected graph and $e$ is an edge of $G$. Then $G$ has no pair of edge-disjoint cycles whose union contains $e$ if and only if, for some integer $n\ge 1$, $G$ has subgraphs $G_1,\dots, G_n$ and an $(n+1)$-elements set of vertices $U:=\{u_0,\dots,u_n\}$ such that
\begin{enumerate}
 \item [(a)] $G=(G_1\cup \cdots\cup G_n)+e$ and $e\notin E(G_i)$ for $i=1,\dots,n$;
 \item [(b)] $e=u_0u_n$;
 \item [(c)] for each $0\le i<j\le n$, $u_iu_j\notin E(G)$ if $\{i,j\}\neq\{0,n\}$; 
 \item [(d)] for each $i=1,\dots,n$, $V(G_i)\cap U=\{u_{i-1},u_i\}$;
 \item [(e)] for $1\le i<j\le n$, $V(G_i)\cap V(G_j)=\emptyset$ if $j>i+1$ and $V(G_i)\cap V(G_{i+1})=\{u_i\}$; and
 \item [(f)] for each $i=1,\dots,n$, one of the following assertions holds:
 \begin{enumerate}
  \item [(f1)] $G_i$ is connected with $V(G_i)=\{u_{i-1},u_i\}$,
  \item [(f2)] $G_i$ is a cycle or
  \item [(f3)] $G_i+u_{i-i}u_i$ is a subdivision of a $3$-connected graph with no pair of edge-disjoint cycles whose union contains $u_{i-1}u_i$.
 \end{enumerate}
 \end{enumerate}
\end{theorem}

Next we state a shorter form of Theorem \ref{main-cor}. The \defin{hub} of a wheel graph $W_n$ with $n \geq 4$ spokes is the vertex of degree $n$. If $w$ is a vertex of a $3$-connected graph $G$, then a $3$-connected graph $H$ with an edge $e=uv$ is said to be obtained from $G$ by \defin{splitting} the vertex $w$ by $e$ if the contraction of the edge $e$ from $H$ is the graph $G$ with the vertices $u$ and $v$ of $H$ contracting to form the new vertex $w$ of $G$.

\begin{theorem}\label{main}
Let $G$ be a $3$-connected graph with at least six vertices and $e$ be an edge of $G$. Then $G$ contains no pair of vertex-disjoint cycles whose union contains $e$ if and only if one of the following assertions hold:
\begin{enumerate}
 \item [(a)] for some $n\ge 4$, $G$ is obtained from a wheel graph with $n$ spokes by possibly doubling some spokes and splitting the hub by the edge $e$ so that the resulting graph is simple and $3$-connected and both vertices incident to $e$ have degree at least four;
 \item [(b)] for $e=uv$, $G$ has a vertex $w$ such that $\{u,v,w\}$ is a vertex-cut of $G$ and each connected component $K$ of $G\del \{u,v,w\}$ is a tree with an unique neighbor of $w$ in respect to $G$; or
 \item [(c)] $G$  is a $3$-connected minor  of a graph of \defin{type} (c1), (c2) or (c3) using the edge $e$ as shown in Figures \ref{fig-type1}, \ref{fig-type2},  and \ref{fig-type3}.
 \end{enumerate}
\end{theorem}
\begin{center}
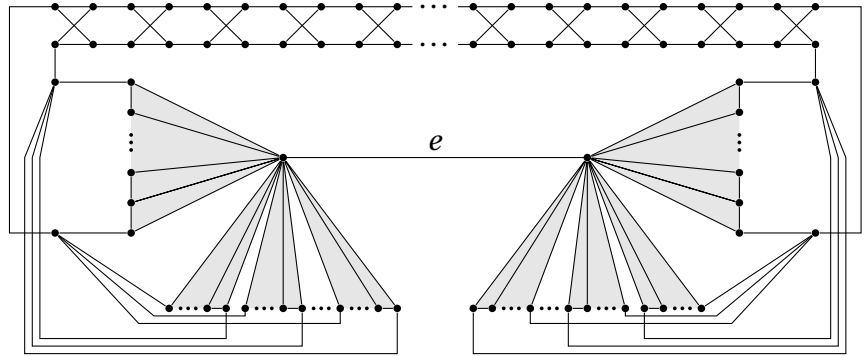
\begin{figure}[h]\begin{tikzpicture}\centering[scale=3]
\tikzstyle{node_style2} = [shape = circle,fill = black,minimum size = 0.5pt,inner sep=0.5pt]
\tikzstyle{node_style} =[shape = circle,fill = black,minimum size = 2pt,inner sep=1pt]
\begin{scope}
	\fill[gray!20,opacity=1] (2,0) -- (0.5,-2) -- (1.25,-2) -- cycle;
	\fill[gray!20,opacity=1] (2,0) -- (1.75,-2) -- (2.5,-2) -- cycle;
	\fill[gray!20,opacity=1] (2,0) -- (2.75,-2) -- (3.5,-2) -- cycle;
	\fill[gray!20,opacity=1] (2,0) -- (4,-1) -- (4,1);
	
	\node[node_style] (v) at (2,0) {};	\node[node_style] (x1) at (4,1) {};\node[node_style] (x2) at (5,1) {};
	\node[node_style] (y1) at (4,-1) {};\node[node_style] (y2) at (5,-1) {};
	
	\node[node_style2] () at (0.9,-2) {};
	\node[node_style2] () at (1,-2) {};
	\node[node_style2] () at (1.1,-2) {};
	\node[node_style] (lfl) at (0.5,-2)  {};\node[node_style] (lfm) at (0.75,-2) {}; \node[node_style] (lfr) at (1.25,-2) {};
	\node[node_style2] () at (1.4,-2){};
	\node[node_style2] () at (1.5,-2){};
	\node[node_style2] () at (1.6,-2){};
	\node[node_style2] () at (2.15,-2){};
	\node[node_style2] () at (2.25,-2){};
	\node[node_style2] () at (2.35,-2){};
	\node[node_style] (mfl) at (1.75,-2) {};\node[node_style] (mfm) at (2,-2) {}; \node[node_style] (mfr) at (2.5,-2) {};
	\node[node_style2] () at (3.15,-2){};
	\node[node_style2] () at (3.25,-2){};
	\node[node_style2] () at (3.35,-2){};
	\node[node_style] (rfl) at (2.75,-2) {};\node[node_style] (rfm) at (3,-2) {}; \node[node_style] (rfr) at (3.5,-2) {};
	
	\draw (y2)--(y1)--(v)--(x1)--(x2);
	\draw (v)--(lfl)--(lfm)--(v)--(lfr);
	\draw (v)--(mfl)--(mfm)--(v)--(mfr);
	
	\draw (v)--(rfl)--(rfm)--(v)--(rfr);
	
	\node[node_style] (1) at (4,-0.6){};
	\node[node_style] (2) at (4,-0.2){};
	\node[node_style] (3) at (4,0.6){};
	\node[node_style2] () at (4,0.1){};
	\node[node_style2] () at (4,0.2){};
	\node[node_style2] () at (4,0.3){};
	\draw (y1)--(1)--(v)--(2)--(1)--(v)--(3)--(x1);	
	
	\draw (y2)--(rfr);
	\draw (y2)--(3.7,-2.1)--(2.5,-2.1)--(2.5,-2);
	\draw (y2)--(3.9,-2.2)--(1.25,-2.2)--(1.25,-2);
	\draw (x2)--(5.2,0)--(5.2,-2.4)--(2.75,-2.4)--(rfl);
	\draw (x2)--(5.3,0)--(5.3,-2.5)--(1.75,-2.5)--(mfl);
	\draw (x2)--(5.4,0)--(5.4,-2.6)--(0.5,-2.6)--(lfl);
	
	\node[node_style] (x3) at (5,1.5){};\node[node_style] (x4) at (4.5,1.5){};
	\node[node_style] (y3) at (5,2){};\node[node_style] (y4) at (4.5,2){};
	\draw (x3)--(y4);\draw(y3)--(x4);
	\node[node_style] (x5) at (4,1.5){};\node[node_style] (x6) at (3.5,1.5){};
	\node[node_style] (y5) at (4,2){};\node[node_style] (y6) at (3.5,2){};
	\draw (x5)--(y6);\draw(y5)--(x6);
	\node[node_style] (x7) at (3,1.5){};\node[node_style] (x8) at (2.5,1.5){};
	\node[node_style] (y7) at (3,2){};\node[node_style] (y8) at (2.5,2){};
	\draw (x7)--(y8);\draw(y7)--(x8);
	\node[node_style] (x9) at (2,1.5){};\node[node_style] (x10) at (1.5,1.5){};
	\node[node_style] (y9) at (2,2){};\node[node_style] (y10) at (1.5,2){};
	\draw (x9)--(y10);\draw(y9)--(x10);
	\node[node_style] (x11) at (1,1.5){};\node[node_style] (x12) at (0.5,1.5){};
	\node[node_style] (y11) at (1,2){};\node[node_style] (y12) at (0.5,2){};
	\draw (x11)--(y12);\draw(y11)--(x12);
	
	\draw (x2)--(x3)--(0.3,1.5);
	\draw (y2)--(5.6,-1)--(5.6,2)--(0.3,2);
\end{scope}

\node (dots) at (0,1.5) {$\dots$};
\node (dots) at (0,2) {$\dots$};
\node (e) at (0,0.2) {$e$};
\draw (-2,0)--(2,0);

\begin{scope}[xscale=-1,yscale=1]
	\fill[gray!20,opacity=1] (2,0) -- (0.5,-2) -- (1.25,-2) -- cycle;
	\fill[gray!20,opacity=1] (2,0) -- (1.75,-2) -- (2.5,-2) -- cycle;
	\fill[gray!20,opacity=1] (2,0) -- (2.75,-2) -- (3.5,-2) -- cycle;
	\fill[gray!20,opacity=1] (2,0) -- (4,-1) -- (4,1);
	
	\node[node_style] (v) at (2,0) {};
	\node[node_style] (x1) at (4,1) {};\node[node_style] (x2) at (5,1) {};
	\node[node_style] (y1) at (4,-1) {};\node[node_style] (y2) at (5,-1) {};
	
	\node[node_style2] () at (0.9,-2) {};
	\node[node_style2] () at (1,-2) {};
	\node[node_style2] () at (1.1,-2) {};
	\node[node_style] (lfl) at (0.5,-2)  {};\node[node_style] (lfm) at (0.75,-2) {}; \node[node_style] (lfr) at (1.25,-2) {};
	\node[node_style2] () at (1.4,-2){};
	\node[node_style2] () at (1.5,-2){};
	\node[node_style2] () at (1.6,-2){};
	\node[node_style2] () at (2.15,-2){};
	\node[node_style2] () at (2.25,-2){};
	\node[node_style2] () at (2.35,-2){};
	\node[node_style] (mfl) at (1.75,-2) {};\node[node_style] (mfm) at (2,-2) {}; \node[node_style] (mfr) at (2.5,-2) {};
	\node[node_style2] () at (3.15,-2){};
	\node[node_style2] () at (3.25,-2){};
	\node[node_style2] () at (3.35,-2){};
	\node[node_style] (rfl) at (2.75,-2) {};\node[node_style] (rfm) at (3,-2) {}; \node[node_style] (rfr) at (3.5,-2) {};
	
	\draw (y2)--(y1)--(v)--(x1)--(x2);
	\draw (v)--(lfl)--(lfm)--(v)--(lfr);
	\draw (v)--(mfl)--(mfm)--(v)--(mfr);
	
	\draw (v)--(rfl)--(rfm)--(v)--(rfr);
	
	\node[node_style] (1) at (4,-0.6){};
	\node[node_style] (2) at (4,-0.2){};
	\node[node_style] (3) at (4,0.6){};
	\node[node_style2] () at (4,0.1){};
	\node[node_style2] () at (4,0.2){};
	\node[node_style2] () at (4,0.3){};
	\draw (y1)--(1)--(v)--(2)--(1)--(v)--(3)--(x1);	
	
	\draw (y2)--(rfr);
	\draw (y2)--(3.7,-2.1)--(2.5,-2.1)--(2.5,-2);
	\draw (y2)--(3.9,-2.2)--(1.25,-2.2)--(1.25,-2);
	\draw (x2)--(5.2,0)--(5.2,-2.4)--(2.75,-2.4)--(rfl);
	\draw (x2)--(5.3,0)--(5.3,-2.5)--(1.75,-2.5)--(mfl);
	\draw (x2)--(5.4,0)--(5.4,-2.6)--(0.5,-2.6)--(lfl);
	
	\node[node_style] (x3) at (5,1.5){};\node[node_style] (x4) at (4.5,1.5){};
	\node[node_style] (y3) at (5,2){};\node[node_style] (y4) at (4.5,2){};
	\draw (x3)--(y4);\draw(y3)--(x4);
	\node[node_style] (x5) at (4,1.5){};\node[node_style] (x6) at (3.5,1.5){};
	\node[node_style] (y5) at (4,2){};\node[node_style] (y6) at (3.5,2){};
	\draw (x5)--(y6);\draw(y5)--(x6);
	\node[node_style] (x7) at (3,1.5){};\node[node_style] (x8) at (2.5,1.5){};
	\node[node_style] (y7) at (3,2){};\node[node_style] (y8) at (2.5,2){};
	\draw (x7)--(y8);\draw(y7)--(x8);
	\node[node_style] (x9) at (2,1.5){};\node[node_style] (x10) at (1.5,1.5){};
	\node[node_style] (y9) at (2,2){};\node[node_style] (y10) at (1.5,2){};
	\draw (x9)--(y10);\draw(y9)--(x10);
	\node[node_style] (x11) at (1,1.5){};\node[node_style] (x12) at (0.5,1.5){};
	\node[node_style] (y11) at (1,2){};\node[node_style] (y12) at (0.5,2){};
	\draw (x11)--(y12);\draw(y11)--(x12);
	
	\draw (x2)--(x3)--(0.3,1.5);
	\draw (y2)--(5.6,-1)--(5.6,2)--(0.3,2);
\end{scope}
\end{tikzpicture}
\caption{The type (c1) graphs.}
\label{fig-type1}
\end{figure}

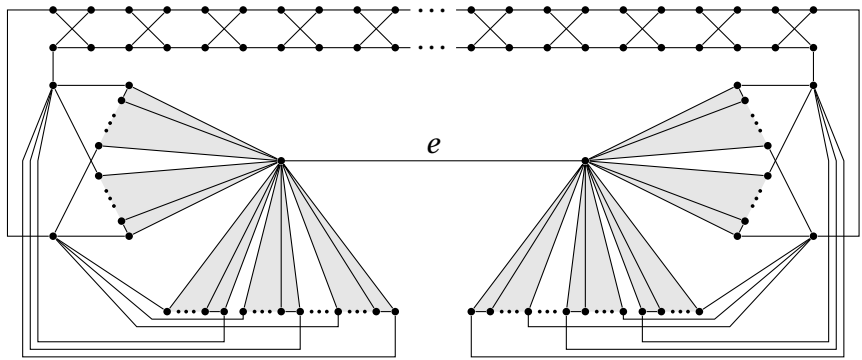
\begin{figure}[h]\begin{tikzpicture}\centering
\tikzstyle{node_style2} = [shape = circle,fill = black,minimum size = 0.5pt,inner sep=0.5pt]
\tikzstyle{node_style} =[shape = circle,fill = black,minimum size = 2pt,inner sep=1pt]
\begin{scope}
	\fill[gray!20,opacity=1] (2,0) -- (0.5,-2) -- (1.25,-2) -- cycle;
	\fill[gray!20,opacity=1] (2,0) -- (1.75,-2) -- (2.5,-2) -- cycle;
	\fill[gray!20,opacity=1] (2,0) -- (2.75,-2) -- (3.5,-2) -- cycle;
	\fill[gray!20,opacity=1] (2,0) -- (4,-1) -- (4.4,-0.2);
	\fill[gray!20,opacity=1] (2,0) -- (4,1) -- (4.4, 0.2);
	
	\node[node_style] (v) at (2,0) {};
	\node[node_style] (x1) at (4,1) {};\node[node_style] (x2) at (5,1) {};
	\node[node_style] (y1) at (4,-1) {};\node[node_style] (y2) at (5,-1) {};
	
	\node[node_style2] () at (0.9,-2) {};
	\node[node_style2] () at (1,-2) {};
	\node[node_style2] () at (1.1,-2) {};
	\node[node_style] (lfl) at (0.5,-2)  {};\node[node_style] (lfm) at (0.75,-2) {}; \node[node_style] (lfr) at (1.25,-2) {};
	\node[node_style2] () at (1.4,-2){};
	\node[node_style2] () at (1.5,-2){};
	\node[node_style2] () at (1.6,-2){};
	\node[node_style2] () at (2.15,-2){};
	\node[node_style2] () at (2.25,-2){};
	\node[node_style2] () at (2.35,-2){};
	\node[node_style] (mfl) at (1.75,-2) {};\node[node_style] (mfm) at (2,-2) {}; \node[node_style] (mfr) at (2.5,-2) {};
	\node[node_style2] () at (3.15,-2){};
	\node[node_style2] () at (3.25,-2){};
	\node[node_style2] () at (3.35,-2){};
	\node[node_style] (rfl) at (2.75,-2) {};\node[node_style] (rfm) at (3,-2) {}; \node[node_style] (rfr) at (3.5,-2) {};
	
	\draw (y2)--(y1)--(v)--(x1)--(x2);
	\draw (v)--(lfl)--(lfm)--(v)--(lfr);
	\draw (v)--(mfl)--(mfm)--(v)--(mfr);
	
	\draw (v)--(rfl)--(rfm)--(v)--(rfr);
	
	\node[node_style] (lofl) at (4.1,-0.8){};\node[node_style] (lofu) at (4.4,-0.2){}; 	\draw (y1)--(lofl)--(v)--(lofu);
	\node[node_style2] (ldot1) at (4.2,-0.6){};
	\node[node_style2] (ldot2) at (4.25,-0.5){};
	\node[node_style2] (ldot3) at (4.3,-0.4){};
	\node[node_style] (upfu) at (4.1,0.8){};\node[node_style] (upfl) at (4.4, 0.2){};	\draw (x1)--(upfu)--(v)--(upfl);
	\node[node_style2] (ldot1) at (4.2,0.6){};
	\node[node_style2] (ldot2) at (4.25,0.5){};
	\node[node_style2] (ldot3) at (4.3,0.4){};
	\draw (lofu)--(x2);
	\draw (upfl)--(y2);
	
	\draw (y2)--(rfr);
	\draw (y2)--(3.7,-2.1)--(2.5,-2.1)--(2.5,-2);
	\draw (y2)--(3.9,-2.2)--(1.25,-2.2)--(1.25,-2);
	\draw (x2)--(5.2,0)--(5.2,-2.4)--(2.75,-2.4)--(rfl);
	\draw (x2)--(5.3,0)--(5.3,-2.5)--(1.75,-2.5)--(mfl);
	\draw (x2)--(5.4,0)--(5.4,-2.6)--(0.5,-2.6)--(lfl);
	
	\node[node_style] (x3) at (5,1.5){};\node[node_style] (x4) at (4.5,1.5){};
	\node[node_style] (y3) at (5,2){};\node[node_style] (y4) at (4.5,2){};
	\draw (x3)--(y4);\draw(y3)--(x4);
	\node[node_style] (x5) at (4,1.5){};\node[node_style] (x6) at (3.5,1.5){};
	\node[node_style] (y5) at (4,2){};\node[node_style] (y6) at (3.5,2){};
	\draw (x5)--(y6);\draw(y5)--(x6);
	\node[node_style] (x7) at (3,1.5){};\node[node_style] (x8) at (2.5,1.5){};
	\node[node_style] (y7) at (3,2){};\node[node_style] (y8) at (2.5,2){};
	\draw (x7)--(y8);\draw(y7)--(x8);
	\node[node_style] (x9) at (2,1.5){};\node[node_style] (x10) at (1.5,1.5){};
	\node[node_style] (y9) at (2,2){};\node[node_style] (y10) at (1.5,2){};
	\draw (x9)--(y10);\draw(y9)--(x10);
	\node[node_style] (x11) at (1,1.5){};\node[node_style] (x12) at (0.5,1.5){};
	\node[node_style] (y11) at (1,2){};\node[node_style] (y12) at (0.5,2){};
	\draw (x11)--(y12);\draw(y11)--(x12);
	
	\draw (x2)--(x3)--(0.3,1.5);
	\draw (y2)--(5.6,-1)--(5.6,2)--(0.3,2);
\end{scope}

\node (dots) at (0,1.5) {$\dots$};
\node (dots) at (0,2) {$\dots$};
\node (e) at (0,0.2) {$e$};
\draw (-2,0)--(2,0);

\begin{scope}[xscale=-1,yscale=1]
	\fill[gray!20,opacity=1] (2,0) -- (0.5,-2) -- (1.25,-2) -- cycle;
	\fill[gray!20,opacity=1] (2,0) -- (1.75,-2) -- (2.5,-2) -- cycle;
	\fill[gray!20,opacity=1] (2,0) -- (2.75,-2) -- (3.5,-2) -- cycle;
	\fill[gray!20,opacity=1] (2,0) -- (4,-1) -- (4.4,-0.2);
	\fill[gray!20,opacity=1] (2,0) -- (4,1) -- (4.4, 0.2);
	
	\node[node_style] (v) at (2,0) {};
	\node[node_style] (x1) at (4,1) {};\node[node_style] (x2) at (5,1) {};
	\node[node_style] (y1) at (4,-1) {};\node[node_style] (y2) at (5,-1) {};
	
	\node[node_style2] () at (0.9,-2) {};
	\node[node_style2] () at (1,-2) {};
	\node[node_style2] () at (1.1,-2) {};
	\node[node_style] (lfl) at (0.5,-2)  {};\node[node_style] (lfm) at (0.75,-2) {}; \node[node_style] (lfr) at (1.25,-2) {};
	\node[node_style2] () at (1.4,-2){};
	\node[node_style2] () at (1.5,-2){};
	\node[node_style2] () at (1.6,-2){};
	\node[node_style2] () at (2.15,-2){};
	\node[node_style2] () at (2.25,-2){};
	\node[node_style2] () at (2.35,-2){};
	\node[node_style] (mfl) at (1.75,-2) {};\node[node_style] (mfm) at (2,-2) {}; \node[node_style] (mfr) at (2.5,-2) {};
	\node[node_style2] () at (3.15,-2){};
	\node[node_style2] () at (3.25,-2){};
	\node[node_style2] () at (3.35,-2){};
	\node[node_style] (rfl) at (2.75,-2) {};\node[node_style] (rfm) at (3,-2) {}; \node[node_style] (rfr) at (3.5,-2) {};
	
	\draw (y2)--(y1)--(v)--(x1)--(x2);
	\draw (v)--(lfl)--(lfm)--(v)--(lfr);
	\draw (v)--(mfl)--(mfm)--(v)--(mfr);
	
	\draw (v)--(rfl)--(rfm)--(v)--(rfr);
	
	\node[node_style] (lofl) at (4.1,-0.8){};\node[node_style] (lofu) at (4.4,-0.2){}; 	\draw (y1)--(lofl)--(v)--(lofu);
	\node[node_style2] (ldot1) at (4.2,-0.6){};
	\node[node_style2] (ldot2) at (4.25,-0.5){};
	\node[node_style2] (ldot3) at (4.3,-0.4){};
	\node[node_style] (upfu) at (4.1,0.8){};\node[node_style] (upfl) at (4.4, 0.2){};	\draw (x1)--(upfu)--(v)--(upfl);
	\node[node_style2] (ldot1) at (4.2,0.6){};
	\node[node_style2] (ldot2) at (4.25,0.5){};
	\node[node_style2] (ldot3) at (4.3,0.4){};
	\draw (lofu)--(x2);
	\draw (upfl)--(y2);
	
	\draw (y2)--(rfr);
	\draw (y2)--(3.7,-2.1)--(2.5,-2.1)--(2.5,-2);
	\draw (y2)--(3.9,-2.2)--(1.25,-2.2)--(1.25,-2);
	\draw (x2)--(5.2,0)--(5.2,-2.4)--(2.75,-2.4)--(rfl);
	\draw (x2)--(5.3,0)--(5.3,-2.5)--(1.75,-2.5)--(mfl);
	\draw (x2)--(5.4,0)--(5.4,-2.6)--(0.5,-2.6)--(lfl);
	
	\node[node_style] (x3) at (5,1.5){};\node[node_style] (x4) at (4.5,1.5){};
	\node[node_style] (y3) at (5,2){};\node[node_style] (y4) at (4.5,2){};
	\draw (x3)--(y4);\draw(y3)--(x4);
	\node[node_style] (x5) at (4,1.5){};\node[node_style] (x6) at (3.5,1.5){};
	\node[node_style] (y5) at (4,2){};\node[node_style] (y6) at (3.5,2){};
	\draw (x5)--(y6);\draw(y5)--(x6);
	\node[node_style] (x7) at (3,1.5){};\node[node_style] (x8) at (2.5,1.5){};
	\node[node_style] (y7) at (3,2){};\node[node_style] (y8) at (2.5,2){};
	\draw (x7)--(y8);\draw(y7)--(x8);
	\node[node_style] (x9) at (2,1.5){};\node[node_style] (x10) at (1.5,1.5){};
	\node[node_style] (y9) at (2,2){};\node[node_style] (y10) at (1.5,2){};
	\draw (x9)--(y10);\draw(y9)--(x10);
	\node[node_style] (x11) at (1,1.5){};\node[node_style] (x12) at (0.5,1.5){};
	\node[node_style] (y11) at (1,2){};\node[node_style] (y12) at (0.5,2){};
	\draw (x11)--(y12);\draw(y11)--(x12);
	
	\draw (x2)--(x3)--(0.3,1.5);
	\draw (y2)--(5.6,-1)--(5.6,2)--(0.3,2);
\end{scope}
\end{tikzpicture}
\caption{The type (c2) graphs.}
\label{fig-type2}
\end{figure}

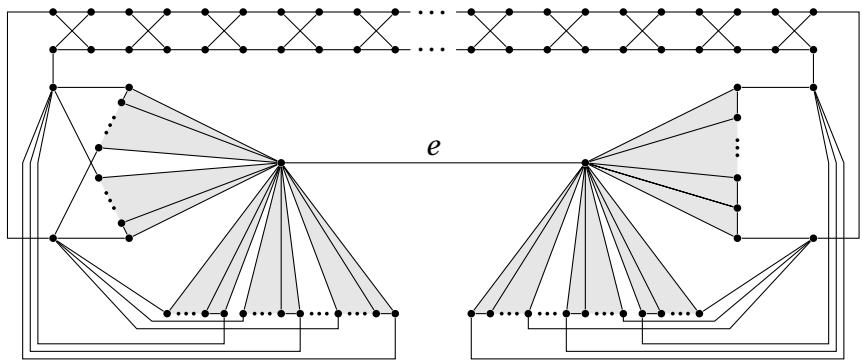
\begin{figure}[h]\begin{tikzpicture}\centering[scale=3]
\tikzstyle{node_style2} = [shape = circle,fill = black,minimum size = 0.5pt,inner sep=0.5pt]
\tikzstyle{node_style} =[shape = circle,fill = black,minimum size = 2pt,inner sep=1pt]
\begin{scope}
	\fill[gray!20,opacity=1] (2,0) -- (0.5,-2) -- (1.25,-2) -- cycle;
	\fill[gray!20,opacity=1] (2,0) -- (1.75,-2) -- (2.5,-2) -- cycle;
	\fill[gray!20,opacity=1] (2,0) -- (2.75,-2) -- (3.5,-2) -- cycle;
	\fill[gray!20,opacity=1] (2,0) -- (4,-1) -- (4,1);
	
	\node[node_style] (v) at (2,0) {};
	\node[node_style] (x1) at (4,1) {};\node[node_style] (x2) at (5,1) {};
	\node[node_style] (y1) at (4,-1) {};\node[node_style] (y2) at (5,-1) {};
	
	\node[node_style2] () at (0.9,-2) {};
	\node[node_style2] () at (1,-2) {};
	\node[node_style2] () at (1.1,-2) {};
	\node[node_style] (lfl) at (0.5,-2)  {};\node[node_style] (lfm) at (0.75,-2) {}; \node[node_style] (lfr) at (1.25,-2) {};
	\node[node_style2] () at (1.4,-2){};
	\node[node_style2] () at (1.5,-2){};
	\node[node_style2] () at (1.6,-2){};
	\node[node_style2] () at (2.15,-2){};
	\node[node_style2] () at (2.25,-2){};
	\node[node_style2] () at (2.35,-2){};
	\node[node_style] (mfl) at (1.75,-2) {};\node[node_style] (mfm) at (2,-2) {}; \node[node_style] (mfr) at (2.5,-2) {};
	\node[node_style2] () at (3.15,-2){};
	\node[node_style2] () at (3.25,-2){};
	\node[node_style2] () at (3.35,-2){};
	\node[node_style] (rfl) at (2.75,-2) {};\node[node_style] (rfm) at (3,-2) {}; \node[node_style] (rfr) at (3.5,-2) {};
	
	\draw (y2)--(y1)--(v)--(x1)--(x2);
	\draw (v)--(lfl)--(lfm)--(v)--(lfr);
	\draw (v)--(mfl)--(mfm)--(v)--(mfr);
	
	\draw (v)--(rfl)--(rfm)--(v)--(rfr);
	
	\node[node_style] (1) at (4,-0.6){};
	\node[node_style] (2) at (4,-0.2){};
	\node[node_style] (3) at (4,0.6){};
	\node[node_style2] () at (4,0.1){};
	\node[node_style2] () at (4,0.2){};
	\node[node_style2] () at (4,0.3){};
	\draw (y1)--(1)--(v)--(2)--(1)--(v)--(3)--(x1);	
	
	\draw (y2)--(rfr);
	\draw (y2)--(3.7,-2.1)--(2.5,-2.1)--(2.5,-2);
	\draw (y2)--(3.9,-2.2)--(1.25,-2.2)--(1.25,-2);
	\draw (x2)--(5.2,0)--(5.2,-2.4)--(2.75,-2.4)--(rfl);
	\draw (x2)--(5.3,0)--(5.3,-2.5)--(1.75,-2.5)--(mfl);
	\draw (x2)--(5.4,0)--(5.4,-2.6)--(0.5,-2.6)--(lfl);
	
	\node[node_style] (x3) at (5,1.5){};\node[node_style] (x4) at (4.5,1.5){};
	\node[node_style] (y3) at (5,2){};\node[node_style] (y4) at (4.5,2){};
	\draw (x3)--(y4);\draw(y3)--(x4);
	\node[node_style] (x5) at (4,1.5){};\node[node_style] (x6) at (3.5,1.5){};
	\node[node_style] (y5) at (4,2){};\node[node_style] (y6) at (3.5,2){};
	\draw (x5)--(y6);\draw(y5)--(x6);
	\node[node_style] (x7) at (3,1.5){};\node[node_style] (x8) at (2.5,1.5){};
	\node[node_style] (y7) at (3,2){};\node[node_style] (y8) at (2.5,2){};
	\draw (x7)--(y8);\draw(y7)--(x8);
	\node[node_style] (x9) at (2,1.5){};\node[node_style] (x10) at (1.5,1.5){};
	\node[node_style] (y9) at (2,2){};\node[node_style] (y10) at (1.5,2){};
	\draw (x9)--(y10);\draw(y9)--(x10);
	\node[node_style] (x11) at (1,1.5){};\node[node_style] (x12) at (0.5,1.5){};
	\node[node_style] (y11) at (1,2){};\node[node_style] (y12) at (0.5,2){};
	\draw (x11)--(y12);\draw(y11)--(x12);
	
	\draw (x2)--(x3)--(0.3,1.5);
	\draw (y2)--(5.6,-1)--(5.6,2)--(0.3,2);
\end{scope}

\node (dots) at (0,1.5) {$\dots$};
\node (dots) at (0,2) {$\dots$};
\node (e) at (0,0.2) {$e$};
\draw (-2,0)--(2,0);

\begin{scope}[xscale=-1,yscale=1]
	\fill[gray!20,opacity=1] (2,0) -- (0.5,-2) -- (1.25,-2) -- cycle;
	\fill[gray!20,opacity=1] (2,0) -- (1.75,-2) -- (2.5,-2) -- cycle;
	\fill[gray!20,opacity=1] (2,0) -- (2.75,-2) -- (3.5,-2) -- cycle;
	\fill[gray!20,opacity=1] (2,0) -- (4,-1) -- (4.4,-0.2);
	\fill[gray!20,opacity=1] (2,0) -- (4,1) -- (4.4, 0.2);
	
	\node[node_style] (v) at (2,0) {};
	\node[node_style] (x1) at (4,1) {};\node[node_style] (x2) at (5,1) {};
	\node[node_style] (y1) at (4,-1) {};\node[node_style] (y2) at (5,-1) {};
	
	\node[node_style2] () at (0.9,-2) {};
	\node[node_style2] () at (1,-2) {};
	\node[node_style2] () at (1.1,-2) {};
	\node[node_style] (lfl) at (0.5,-2)  {};\node[node_style] (lfm) at (0.75,-2) {}; \node[node_style] (lfr) at (1.25,-2) {};
	\node[node_style2] () at (1.4,-2){};
	\node[node_style2] () at (1.5,-2){};
	\node[node_style2] () at (1.6,-2){};
	\node[node_style2] () at (2.15,-2){};
	\node[node_style2] () at (2.25,-2){};
	\node[node_style2] () at (2.35,-2){};
	\node[node_style] (mfl) at (1.75,-2) {};\node[node_style] (mfm) at (2,-2) {}; \node[node_style] (mfr) at (2.5,-2) {};
	\node[node_style2] () at (3.15,-2){};
	\node[node_style2] () at (3.25,-2){};
	\node[node_style2] () at (3.35,-2){};
	\node[node_style] (rfl) at (2.75,-2) {};\node[node_style] (rfm) at (3,-2) {}; \node[node_style] (rfr) at (3.5,-2) {};
	
	\draw (y2)--(y1)--(v)--(x1)--(x2);
	\draw (v)--(lfl)--(lfm)--(v)--(lfr);
	\draw (v)--(mfl)--(mfm)--(v)--(mfr);
	
	\draw (v)--(rfl)--(rfm)--(v)--(rfr);
	
	\node[node_style] (lofl) at (4.1,-0.8){};\node[node_style] (lofu) at (4.4,-0.2){}; 	\draw (y1)--(lofl)--(v)--(lofu);
	\node[node_style2] (ldot1) at (4.2,-0.6){};
	\node[node_style2] (ldot2) at (4.25,-0.5){};
	\node[node_style2] (ldot3) at (4.3,-0.4){};
	\node[node_style] (upfu) at (4.1,0.8){};\node[node_style] (upfl) at (4.4, 0.2){};	\draw (x1)--(upfu)--(v)--(upfl);
	\node[node_style2] (ldot1) at (4.2,0.6){};
	\node[node_style2] (ldot2) at (4.25,0.5){};
	\node[node_style2] (ldot3) at (4.3,0.4){};
	\draw (lofu)--(x2);
	\draw (upfl)--(y2);
	
	\draw (y2)--(rfr);
	\draw (y2)--(3.7,-2.1)--(2.5,-2.1)--(2.5,-2);
	\draw (y2)--(3.9,-2.2)--(1.25,-2.2)--(1.25,-2);
	\draw (x2)--(5.2,0)--(5.2,-2.4)--(2.75,-2.4)--(rfl);
	\draw (x2)--(5.3,0)--(5.3,-2.5)--(1.75,-2.5)--(mfl);
	\draw (x2)--(5.4,0)--(5.4,-2.6)--(0.5,-2.6)--(lfl);
	
	\node[node_style] (x3) at (5,1.5){};\node[node_style] (x4) at (4.5,1.5){};
	\node[node_style] (y3) at (5,2){};\node[node_style] (y4) at (4.5,2){};
	\draw (x3)--(y4);\draw(y3)--(x4);
	\node[node_style] (x5) at (4,1.5){};\node[node_style] (x6) at (3.5,1.5){};
	\node[node_style] (y5) at (4,2){};\node[node_style] (y6) at (3.5,2){};
	\draw (x5)--(y6);\draw(y5)--(x6);
	\node[node_style] (x7) at (3,1.5){};\node[node_style] (x8) at (2.5,1.5){};
	\node[node_style] (y7) at (3,2){};\node[node_style] (y8) at (2.5,2){};
	\draw (x7)--(y8);\draw(y7)--(x8);
	\node[node_style] (x9) at (2,1.5){};\node[node_style] (x10) at (1.5,1.5){};
	\node[node_style] (y9) at (2,2){};\node[node_style] (y10) at (1.5,2){};
	\draw (x9)--(y10);\draw(y9)--(x10);
	\node[node_style] (x11) at (1,1.5){};\node[node_style] (x12) at (0.5,1.5){};
	\node[node_style] (y11) at (1,2){};\node[node_style] (y12) at (0.5,2){};
	\draw (x11)--(y12);\draw(y11)--(x12);
	
	\draw (x2)--(x3)--(0.3,1.5);
	\draw (y2)--(5.6,-1)--(5.6,2)--(0.3,2);
\end{scope}
\end{tikzpicture}
\caption{The type (c3) graphs.}
\label{fig-type3}
\end{figure}
\end{center}

A graph as in Theorem~\ref{main} will be called of \label{pg-type}\defin{type} (a), (b), or (c) according to which case of the theorem conclusion it satisfies. A graph may be of more than one type; for instance, the planar type (a) graphs are also type (c) graphs.

For $4$-connected graphs, we have a strengthening of the previous result. A \defin{double-wheel} graph is obtained from a cycle $C$ by adding two new vertices, called \defin{hubs} of the double-wheel, and then adding an edge linking each vertex of $C$ to each one of the hubs. From Theorem \ref{main}, we conclude:

\begin{corollary}\label{4-connected-main}
Let $G$ be a $4$-connected graph with a fixed edge $e$. Then, either $G$ contains a pair of vertex-disjoint cycles using $e$ or $G$ is obtained from a double-wheel by adding the edge $e$ between the two hubs.
\end{corollary}
\begin{proof}Let $G$ be a $4$-connected graph with and edge $e$ such that no pair of vertex-disjoint cycles has the property that $e$ is in one of them. So $G$ is a graph of type (a), (b) or (c) like in Theorme \ref{main}. First suppsoe that $G$ is a type (a) graph, obtained from a wheel as described in Theore \ref{main}. If all spokes are doubed, $G$ is a double wheel, then $G$ is obtained from a double wheel adding $e$ linking the hubs. If some spoke is not doubled, then it is incident to a degree $3$-vertex of $G$ and $G$ is not $3$-connected. As no type (b) graph is $4$-connected, to finish the proof it suffices to prove that type (c) graphs are also not $4$-connected. This follows from the fact that if $C$ is a $k$-vertex cut in a $k$-connected graph $G$ and $X\cont E(G)$, then either the vertices of $G/X$ inherited from $C$ separate $G$ or all but one connected components of $G\del C$ are eliminated when contracting $X$. It is a straighforward verification to use this fact to deduce that the contraction of any set of edges avoiding $e$ in the graphs of Figures \ref{fig-type1}, \ref{fig-type2} or \ref{fig-type3} will result in a non $4$-connected graph.
\end{proof}


 One last remark is that the analogous questions for disjoint cycles containing a specified vertex is reduced to Dirac's characterization, as we can see in the next proposition.
\begin{proposition}
Let $G$ be a $3$-connected graph and $v\in V(G)$. If $G$ has no pair of vertex-disjoint cycles using $v$, then $G$ has no pair of vertex-disjoint cycles.
\end{proposition}
\begin{proof}
Suppose that $G$ and $v$ contradict the proposition. Let $C$ and $D$ be vertex-disjoint cycles of $G$. So $v\notin V(C)\cup V(D)$. By Menger's Theorem, there are three $(v,V(C)\cup V(D))$-paths in $G$ meeting only in $v$. We may assume that two of them have an endvertex in $C$. Note that the union of $C$ with those two paths contains a cycle containing $v$ and avoiding $D$, a contradiction.
\end{proof}

The structure of this paper is described next. For this description, we let $e=u_1u_2$ be an edge in a $3$-connected graph $G$ that is not contained in the union of two vertex-disjoint cycles. In Section \ref{sec-lemmas}, we give some terminologies, establish some lemmas, prove that the endvertices of $e$ are in a $3$- or $4$-vertex cut of $G$ and characterize family (a) of Theorem \ref{main}. In Section \ref{sec-one-two}, we introduce, in the remaining cases, a classification for the $4$-vertex-cuts containing $\{u_1,u_2\}$ and characterize the connected components for the removal of a certain type of a $4$-vertex-cut containing $\{u_1,u_2\}$. For the remaining case, in Section \ref{sec-rope bridges}, we define constructively a type of graph called rope-bridge and prove that it is the remaining type of connected component for the removal of a $3$- or $4$-vertex-cut incident to $e$. In the end of Section \ref{sec-rope bridges}, we prove Theorem \ref{main-cor}, which implies Theorem \ref{main}. Section \ref{sec-strong} is dedicated to proving Theorems \ref{main-strong} and \ref{main-strong-2con}. In Section \ref{sec-prism} we prove Theorem \ref{thm-prism}.

\section{Lemmas}\label{sec-lemmas}
We first establish some terminology before giving useful lemmas for the proof of the main results. The paths we consider are simple, we think of paths and cycles both as subgraphs and (cyclic) sequences of vertices. The number of vertices and edges in a graph $G$ are denoted by $|G|$ and $\|G\|$, respectivelly. For vertices $v_1,\dots,v_n$ and a subgraph $H$ of a graph $G$, we say that a \defin{cycle of the form} $v_1,\dots,v_k,H,v_{k+1},\dots,v_n,v_1$ is a cycle that begins in $v_1$, follows through $v_1,\dots,v_k$, then, through a path of $H$, and, then, returns to $v_1$ through $v_{k+1},\dots,v_n$.  We  simplify the notations $X\cup\{x\}$ and $X-\{x\}$ by $X\cup x$ and $X-x$ respectively.

Let $e$ be a fixed edge of a graph $G$. We say that $G$ is an \defin{$e$-Dirac} graph if $G$ contains no vertex-disjoint cycles using $e$, in this case $e$ is a \defin{Dirac} edge of $G$. If $G$ contains no edge-disjoint cycles using $e$, then $G$ is said to be \defin{strongly $e$-Dirac}. An elementary observation about the class of $e$-Dirac graphs is that it is closed under minors that use $e$. That is, if $e$ is a Dirac edge of $G$ and $f\in E(G)-e$, then $e$ is also a Dirac edge of $G\del f$ and $G/f$. But the class of strongly $e$-Dirac graphs is closed under deletions but not under contractions of other edges than $e$. For example, when $G$ is the \defin{prism}, a graph on six vertices obtained by adding a perfect matching between the vertices of two disjoint triangles, and $e$ and $f$ are distinct edges in this perfect matching, then $G$ is strongly $e$-Dirac, but $G/f$ is not. 

The proof of the next lemma is straightforward.

\begin{lemma}\label{complete}
Let $e=u_1u_2$ be an edge in a $3$-connected graph $G$ and suppose that $\{u_1,u_2\}$ is contained in no vertex-cut of $G$. Then $G\del\{u_1,u_2\}$ is complete. Moreover, if $G$ is $e$-Dirac, then $|G|\le 5$.
\end{lemma}

Throughout the paper, until the proof of Theorem \ref{main}, we let $G$ be a $3$-connected graph with a Dirac edge $e=u_1u_2$ and a minimum sized set $S$ such that $S^+:=S\cup\{u_1,u_2\}$ is a vertex-cut of $G$. We let $G_1,\dots,G_\kappa$ be the connected components of $G\del S^+$. For a subgraph $H$ of $G$ and $X\cont V(G)-V(H)$ we denote by $H+X$ the graph obtained from $H$ by adding the vertices of $X$ and all edges of $G$ linking the vertices of $X$ with vertices of $V(H)\cup X$. 

The minimality of $S$ implies the following elementary result.

\begin{lemma}\label{edges from S}
For each $w\in S$ and each $1 \leq i \leq \kappa$, there is an edge from $w$ to a vertex of $G_i$. 
\end{lemma}


\begin{lemma}\label{neighbor}
Each vertex $x\in V(G)-\{u_1,u_2\}$ has at least $|S|$ neighbors in $G\backslash \{u_1,u_2\}$.
\end{lemma}
\begin{proof}Suppose that $x$ has $n\le|S|-1$ neighbors in $G\backslash \{u_1,u_2\}$. Thus $X:=N_G(x)\cup\{x,u_1,u_2\}$ has $n+3\le |S|+2$ elements. As $S^+$ is a vertex-cut, then $|G|\ge |S|+4$. Hence $V(G)-X$ has a vertex $v$. But $N_G(x)\cont X-x$. So, $X-x$ separates $x$ from $v$ and, therefore, $X-x$ is a vertex-cut containing $\{u_1,u_2\}$. But $|X-x|=n+2\le|S|+1<|S^+|$, a contradiction to the minimality of $S$.
\end{proof}

The next lemma has an elementary proof, which we ommit.

\begin{lemma}\label{starbucks}
Let $D$ be a cycle of $G$ such that $e\in E(D)$ and $H$ be a subgraph of $G$ such that $H\del V(D)$ is connected. Suppose that $f$ is an edge of $G$ incident to $x\in V(H)-V(D)$ but such that $f\notin E(H)$. Then, each $(x,V(H)-x)$-path $\gamma$ of $G$ beginning with $f$ intersects $D$. In particular, $\gamma$ has an endvertex in $V(D)$ if $V(D)\cont V(H)$.
\end{lemma}


\begin{lemma}\label{menger}
Let $H$ be a $3$-connected graph, $uv$ be an edge of $H$ and $X$ be a subset of $V(H)$ avoiding $\{u,v\}$ with at least $3$ elements. Then, for some $x\in\{u,v\}$, there are $(X,\{u,v\})$-paths $\alpha$, $\beta$, and $\gamma$ of $G$ such that $V(\alpha)\cap V(\gamma)=V(\beta)\cap V(\gamma)=\emptyset$ and $V(\alpha)\cap V(\beta)=\{x\}$.
\end{lemma}
\begin{proof}
By Menger's Theorem there are $(X,\{u,v\})$-paths $\alpha_1$, $\alpha_2$ and $\alpha_3$ such that $V(\alpha_i)\cap V(\alpha_j)\cont\{u,v\}$ for $1\le i<j\le3$. We may assume that $V(\alpha_i)\cap\{u,v\}=\{v\}$ for $i=1,2,3$. This implies that $d_H(v)\ge 4$. Consider a graph $K$ with minimum degree at least $3$ obtained from $H$ by splitting $v$ into vertices $v_1$ and $v_2$. It is well known and easy to check that $K$ is $3$-connected. By Menger's Theorem, $K$ has three vertex-disjoint $(X,\{u,v_1,v_2\})$-paths. The desired paths are the corresponding paths in $H$. 
\end{proof}

\begin{lemma}\label{big-cycle}
Let $u$ be a vertex in a vertex-cut $X$ of a graph $H$. If $u$ has $3$ neighbors in some connected component $K$ of $H\del X$, then $H[V(K)\cup u]$ has a cycle containing $u$ with more than $3$ vertices.
\end{lemma}
\begin{proof}
Let $x,y,z$ be distinct neighbors of $u$ in $K$. As $K$ is connected, there is an $(x,y)$-path and an $(y,z)$-path in $K$. If one of these paths has more than one edge, we are done. Otherwise, both have one edge and $u,x,y,z,u$ is the cycle we seek.
\end{proof}


\begin{lemma}\label{bigone}
If $|S|\ge 3$, $d_G(u_1),d_G(u_2)\ge4$, and, for some $\{i,j\}=\{1,2\}$ and $k\in\{1,\dots,\kappa\}$, there is a pair of edges from $u_i$ to $G_k$, then:
\begin{enumerate}

\item [(a)] $\kappa=2$ and
\item [(b)] $N_G(u_j)\cont V(G_{3-k})$.
\end{enumerate}
\end{lemma}
\begin{proof} We may assume that $k=i=1$. Note that there is a cycle $C$ in $G_1+u_1$ meeting $u_1$. Pick $C$ maximizing $|C|$. By the minimality of $S$, $G\del \{u_1,u_2\}$ is $3$-connected. Using Menger's Theorem if $|V(C)-u_1|\ge 3$ and Lemma \ref{menger} if $|V(C)-u_1|=2$, we conclude that there are $(S,V(C)-u_1)$-paths $\alpha_1$, $\alpha_2$ and $\alpha_3$ in $G\del\{u_1,u_2\}$ such that $V(\alpha_i)\cap V(\alpha_3)=\emptyset$ for $i=1,2$ and $V(\alpha_1)\cap V(\alpha_2)$ is empty if $|V(C)-u_1|\ge3$ and a singleton set contained in $V(C)$ otherwise. Note that the inner vertices of these paths are in $G_1$. For $i=1,2,3$, we let $v_i$ and $a_i$ be the endvertices of $\alpha_i$ in $S$ and $V(C)-u_1$ respectively. We write $C$ as a closed path of the form $u_1,\beta_1,a_1,\beta_{21},a_2,\beta_{23},a_3,\beta_3,u_1$ with $\|\beta_{1}\|,\|\beta_{23}\|,\|\beta_{3}\|\ge 1$ and $\|\beta_{21}\|\ge 1$ if and only if $|C|\ge 4$. In particular, we choose $\alpha_1$, $\alpha_2$ and $\alpha_3$ minimizing $\|\beta_1\|$.

First we prove (a). Suppose for a contradiction that $\kappa\ge 3$. If there is an edge from $u_2$ to $G_1$, then, for distinct $u,v\in S$, $G_1+\{u_1,u_2\}$ has a cycle containing $e$ avoiding a cycle of $G_2\cup G_3+\{u,v\}$, what is a contradiction. Thus, there is no edge from $u_2$ to $G_1$. As $d_G(u_2)\ge 4$, there is an edge linking $u_2$ and $G_2+(S-\{v_2,v_1\})$. Now, we have a cycle of the form
\begin{linenomath*}
\[u_1,u_2,G_2+(S-\{v_2,v_1\}),v_3,\alpha_3,a_3,\beta_3,u_1\]
\end{linenomath*}
containing $e$ and avoiding a cycle of the form \begin{linenomath*}\[a_1,\beta_{21},a_2,\alpha_2,v_2,G_3,v_1,\alpha_1,a_1,\]
\end{linenomath*}
a contradiction. So, $\kappa=2$ and (a) holds.

Now suppose that (b) fails. Choose a counter-example maximizing $|N_G(u_i)\cap V(G_k)|$; this choice is subject to the minimality of $\|\beta_1\|$, which is subject to the maximality of $|C|$. Note that we may still assume that $k=i=1$.

Define $X:=C\cup\alpha_1\cup\alpha_2\cup\alpha_3$. First we check:

\begin{rot}\label{bigone-1}
If $\zeta$ is an $(u_2,V(X)-u_1)$-path of $G_1+(S\cup u_2)$, then either
\begin{enumerate}
 \item [(i)] $|C|=3$ and $\zeta$ has an  endvertex in $V(\alpha_1)\cup V(\alpha_2)$ or
\item [(ii)] $|C|\ge 4$ and $\zeta$ has an endvertex in $V(\alpha_2)$.
\end{enumerate}
\end{rot}
\begin{rotproof} Let $u_2$ and $z$ be the endvertices of $\zeta$. First we check that $z\in A:=V(\alpha_1)\cup V(\beta_{21})\cup V(\alpha_{2})$. If not, $z$ is a vertex of $X\del A$, which is a connected graph containing $u_1$. This implies that a cycle of $A\cup G_2$ avoids a cycle of the form $u_1,u_2,\zeta,z,X\del A, u_1$, what is a contradiction. So, $z\in A$.

If $|C|=3$, $\|\beta_{21}\|=0$ and $A=V(\alpha_1)\cup V(\alpha_2)$. So, (i) holds.

Now, assume that $|C|\ge 4$. As $z\in A$, analogously, $z\in B:=V(\alpha_3)\cup V(\beta_{23})\cup V(\alpha_{2})$. Therefore, $z\in A\cap B=V(\alpha_2)$. So, (ii) holds. 
\end{rotproof}

As (b) fails, $u_2$ has neighbors in $G_1\cup S$ and $G_1+(S\cup u_2)$ has an $(u_2,V(X)-u_1)$-path $\delta$, which we may assume that has an endvertex  $d$ of $\alpha_2$ by \ref{bigone-1} ($\alpha_1$ and $\alpha_2$ play similar roles if $|C|=3$). Pick $\delta$ minimizing $|V(\delta)\cap S|$. Note that this choice implies that either $u_2$ has some neighbor in $G_1$ and $V(\delta)\cap S\emptyset$ or $v_2$ has no neighbor in $G_1$ and $V(\delta)\cap S$  has an unique vertex, which is a neighbor of $u_2$ in $S$.

For each $i=1,3$ and each $(u_2,V(X)-u_1)$-path $\zeta$ of $G_1+(S\cup u_2)$ with an endvertex $z$ in $\alpha_2$, we define the following cycle containing $e$:
\begin{linenomath*}
\[C_{\zeta,i}:=u_1,u_2,\zeta,z,\alpha_2,a_2,\beta_{2i},a_i,\beta_i,u_1.\]
\end{linenomath*}

Now we prove:

\begin{rot}\label{bigone-2}
$v_1$ has two neighbors in $V(G_1)$.
\end{rot}
\begin{rotproof}
Suppose the contrary. By the assumption that $\delta$ arrives in $\alpha_2$, $v_1$ avoids $\delta$. If $v_1$ has two neighbors in $V(G_2 + (S-V(\delta))$, then $C_{1,\delta}$ avoids a cycle of $G_2+(S-V(\delta)$, a contradiction. Thus $v_1$ has only one neighbor in $V(G_2 + (S-\int(\delta))$. As $v_1$ has an unique neighbor in $G_1$, it follows by Lemma \ref{neighbor} that  $v_1$ has two neighbors in $G_2+S$, and, therefore, a neighbor $v\in V(\delta)\cap S$. As observed before, $v$ is the unique vertex of $S$ in $\delta$, and, as a consequence, $u_2$ has no neighbor in $G_1$ by the minimality of $|V(\delta)\cap S|$. As $d_g(u_2)\ge 4$, $u_2$ has a neighbor $w\in V(G_2)\cup (S-\{v,v_1\})$. Since $\delta$ arrives at $X$ in $\alpha_2$, it follows that $v_3\notin V(\delta)$. Thus the cycle $v_1,v,\delta,d,\alpha_2,a_2,\beta_{21},a_1,\alpha_1,v_1$ avoids a cycle of the form $u_1,u_2,w,G_2,v_3,\alpha_3,a_3,\beta_3,u_1$, a contradiction.
\end{rotproof}

By \ref{bigone-2} $v_1$ has a neighbor $v'_1\in V(G_1)$ such that $v_1v'_1\notin E(\alpha_1)$. Let $\sigma$ be a $(v_1,V(X)\cup\int(\delta))$-path of $G_1+v_i$ beginning with $v_1v'_1$ and let $s$ be the other endvertex of $\sigma$. 

Let $H_0$ be the subgraph of $G$ obtained from $G_2\cup X$ by adding the dges that connects $X$ to $G_2$. Note that $H_0\del V(C_{\delta,1})$ is connected. By Lemma \ref{starbucks}, $s\in V(C_{\delta,1})$. If $s\in \int(\beta_1)$, then $\sigma$, $\alpha_2$ and $\alpha_3$ contradict the choice of $\alpha_1$, $\alpha_2$ and $\alpha_3$ minimizing $\|\beta_1\|$. So, $s\notin \int(\beta_1)$ and $\sigma$ avoids $\int(\beta_1)$. As $(H_0\del \int(\beta_1))\del V(C_{\delta,3})$ is connected, it follows  from Lemma \ref{starbucks} that $s\in V(C_{\delta,3})$. So,
\begin{linenomath*}
\[s\in V(C_{\delta,1})\cap V(C_{\delta,3}) \cap V(G_1)\cont V(a_2,\alpha_2,d)\cup \int(\delta).\]
\end{linenomath*}
Next we define a path $\tau$ as $\tau:=\sigma$ if $s\in V(\alpha_2)$ and $\tau:=\sigma,s,\delta,d$ otherwise. In the later case by the minimality of $|V(\delta)\cap S|$, once $\delta$ meets $s\in V(G_1)$, $\delta$ does not return to $V(G_2)\cup V(S)$. So, $\tau$ is an $(v_1,V(\alpha_2))$-path of $G_1+v_1$ in all cases.  Let $t$ and $v_1$ be the endvertices of $\tau$. We have that $B:=v_1,\tau,t,\alpha_2,a_2,\beta_{21}a_1,\alpha_1,v_1$ is a cycle of $G_1+v_1$. 

As $|S|\ge 3$, then by the minimality of $|S|$, $G\del\{u_1,u_2\}$ is $3$-connected. As $u_2$ has degree at least $3$ in $G\del u_1$, $G\del u_1$ is $3$-connected. So $G\del \{u_1,v_1\}$ has two $(u_2,V(X)\cap V(G_1))$-paths $\zeta_1$ and $\zeta_2$ meeting only in $u_2$. Let $H_1:=H_0+S$. Note that $H_1\del V(B)$ is a connected graph containing $u_1$. If $\zeta_k$ meets $H_1\del V(B)$ before meeting $B$, then a cycle of the form $u_1,u_2,\zeta_k,H_1\del V(B),u_1$ avoids $B$, a contradiction. So, $\zeta_k$ does not meet $H_1\del V(B)$ before meeting $B$ and, therefore, $\zeta_1$ and $\zeta_2$ are paths of $G_1+\{u_2,v_1\}$. So one of these paths is in $G_1+u_2$ and is a possibility for the choice of $\delta$. By the minimality of $|V(\delta)\cap S$, $\delta$ is a path of $G_1+u_2$.

Next we check that $|C|\ge4$. Suppose the contrary. By the maximality of $|C|$ and by Lemma \ref{big-cycle}, $u_1$ has at most two neighbors in $G_1$ and, by the maximality of $|N_G(u_i)\cap V(G_k)|$, for each $i,k\in\{1,2\}$, $u_i$ has at most two neighbors in $G_k$. So, since $u_1$ and $u_2$ have degree at least four, it follows that there are $w_1,w_2\in V(G_2)\cup S$ such that $u_1w_1,u_2w_2\in E(G)$. Let $v\in S-\{w_1,w_2\}$. If $v$ has two different neighbors in $G_1$, then a cycle of $G_1+v$ avoids a cycle of the form $u_1,u_2,w_2,G_2,w_1,u_1$, a contradiction. So, by Lemma \ref{neighbor}, $v$ has two different neighbors in $G_2+S$, which has a cycle. But this cycle avoids $C_{\delta,1}$ since $\delta$ is a path of $G_1+u_2$, a contradiction. Hence, $|C|\ge 4$. This implies that $a_1\neq a_2$.

If $\zeta_k$ intersects $\int(\tau)$, then cycles of the forms 
\[u_1,u_2,\zeta_k,\tau,v_1,\alpha_1,a_1,\beta_1,u_1\qquad\text{ and}\]
\[v_3,G_2,v_2,\alpha_2,a_2,\beta_{23},a_3,\alpha_3,v_3\] avoid each other, a contradiction. So, $\zeta_k$ does not intersect $\int(\tau)$. By \ref{bigone-1}, the endvertex $z_k$ of $\zeta_k$ in $X$ is in $\alpha_2$. Say that $a_2$, $z_1$, $z_2$ and $v_2$ appear in this order in $\alpha_2$. As $\zeta_k$ has an endvertex in $B$, then $t\in V(z_2,\alpha_2,v_2)$. This implies that $\tau$ avoids $C_{\zeta_1,1}$. Therefore, $C_{\zeta_1,1}$ avoids a cycle of the form $v_1,\tau,t,\alpha_2,v_2,G_2,v_1$, a contradiction.
\end{proof}

\begin{lemma}\label{Sle2}
If $d_G(u_1),d_G(u_2)\ge 4$, then $|S|\le 2$.
\end{lemma}
\begin{proof}\setcounter{rot}{0}
Suppose the contrary. 

\begin{rot}\label{Sle2-1}
There is no $i\in\{1,2\}$ and $k\in\{1,\dots,\kappa\}$ such that there are two edges from $u_i$ to $G_k$.
\end{rot}
\begin{rotproof}
Suppose for a contradiction that there are two edges linking $u_i$ to $G_k$. Say $i=k=1$. By Lemma \ref{bigone}, $N_G(u_2)\cont V(G_2)$. But this implies that the hypothesis of Lemma \ref{bigone} also holds for $i=k=2$, then $N_G(u_1)\cont V(G_1)$. For $l=1,2$, as $d_G(u_l)\ge 4$, then, by Lemma \ref{big-cycle}, there is a cycle $C_l$ of $G_l+u_l$ containing $u_l$ with more than $3$-vertices. By the minimality of $|S|$, $G\del \{u_1,u_2\}$ is $3$-connected. So, there are three vertex-disjoint $(C_1-u_1,C_2-u_2)$-paths in $G\del u_1,u_2$. Together with the path $u_1,u_2$, we have four vertex-disjoint $(C_1,C_2)$-paths. Now it is easy to check that there are vertex-disjoint cycles covering these paths. And, therefore, we have two disjoint cycles with one of them containing $e$, a contradiction.
\end{rotproof}

\begin{rot}\label{Sle2-2}
$u_1$ and $u_2$ have no common neighbor in $S$. 
\end{rot}
\begin{rotproof}
If $w$ contradicts \ref{Sle2-2}, then $u_1,u_2,w,u_1$ is a cycle containing $e$ that, for some $x,y\in S-w$, avoids a cycle of the form $x,G_1,y,G_2,x$.
\end{rotproof}

\begin{rot}\label{Sle2-3}
$\kappa=2$.
\end{rot}
\begin{rotproof}
Suppose the contrary. First we suppose that $u_1v\in E(G)$ for some $v\in S$. If $N_G(u_2)\cont S\cup u_1$, then, as $d_G(u_2)\ge 4$, by \ref{Sle2-2}, it follows that there is a $3$-subset $\{x_1,x_2,x_3\}$ of $N_G(u_2)-\{u_1,v\}$ and cycles of the form  $u_1,u_2,x_1,G_1,v,u_1$ and $x_2,G_2,x_3,G_3,x_2$ avoiding each other, a contradiction. So, we may assume that there is an edge from $u_2$ to $G_1$. But this implies that, for $x,y\in S-v$, cycles of the form $u_1,u_2,G_1,v,u_1$ and $x,G_2,y,G_3,x$ avoid each other, a contradiction. Hence, $u_1$ has no neighbor in $S$ and, analogously, neither has $u_2$. 

If both $u_1$ and $u_2$ have neighbors in a common component of $G\del S^+$, say $G_1$, we have, for  $x,y\in S$, cycles of the forms $u_1,u_2,G_1,u_1$ and $x,G_2,y,G_3,x$ avoiding each other, a contradiction. Therefore, $u_1$ and $u_2$ have no neighbors in $S$ nor in a same connected component of $G\del S^+$. By \ref{Sle2-1}, $\kappa\ge 4$ and we may assume that $u_i$ has a neighbor in $G_i$ for $i=1,2$. Now, for distinct $x,y,z\in S$, cycles of the form, $u_1,u_2,G_2,z,G_1,u_1$ and $x,G_3,y,G_4,x$ avoid each other, a contradiction. 
\end{rotproof}

Let $k\in\{1,2\}$. By \ref{Sle2-1} and \ref{Sle2-3}, $u_k$ has a neighbor $v_k\in S$. Let $x\in S-\{v_1,v_2\}$. If $x$ has two neighbors in $G_1$, then $G_1+x$ has a cycle avoiding a cycle of the form $u_1,u_2,v_2,G_2,v_1,u_1$. Thus, $x$ has only one neighbor in $G_1$ and, analogously, only one neighbor in $G_2$. By Lemma \ref{neighbor}, $x$ has a neighbor $y\in S$. If $y\notin \{v_1,v_2\}$, a cycle of the form $u_1,u_2,v_2,G_1,v_1,u_1$ avoids a cycle of the form $x,y,G_2,x$, a contradiction. So, we may assume that $y=v_1$. As $d_G(u_1)\ge 4$, there is $w\in N_G(u_1)-\{u_2,v_1\}$ and we also may assume that $w\notin V(G_2)$ since $u_1$ has at most one neighbor in $G_2$. Now a cycle of the form $x,y,G_2,x$ avoids a cycle of the form $u_1,u_2,v_2,G_1,w,u_1$, a contradiction.
\end{proof}


\begin{lemma}\label{fan side}
Suppose that $|S|=2$ and, for some connected component $J$ of $G\del S^+$, there are edges from both $u_1$ and $u_2$ to $J$. If $K$ is a connected component other than $J$, then $G[K\cup S]$ is a path with endvertices in $S$, moreover each vertex of $K$ has a neighbor in $\{u_1,u_2\}$.
\end{lemma}
\begin{proof}
Note that $J+\{u_1,u_2\}$ has a cycle containing $e$. Therefore, $G[K\cup S]$ is a tree and so is $K$. This implies that each element of $S$ has an unique neighbor in $K$. Define $S:=\{x,y\}$ and let $v_x$ and $v_y$ be the neighbors of $x$ and $y$ respectively in $K$. Let us check that $v_x$ and $v_y$ are the unique leaves of $K$. Suppose for a contradiction that $l$ is a leave of $K$ different from $v_x$ and $v_y$. This implies that $u_1,u_2,l$ is a cycle of $G$ avoiding a cycle of the form $x,K-l,y,J,x$, a contradiction. So, $v_x$ and $v_y$ are the unique leaves of $K$ and $K$ is an $(v_x,v_y)$-path. Hence, $G[K\cup S]$ is a path with endvertices in $S$. Since each vertex of $G$ has degree at least $3$ and $v_x$ and $v_y$ are the unique neighbors of $x$ and $y$ in $K$ respectively, then each vertex of $K$ has a neighbor in $\{u_1,u_2\}$. This proves the lemma.
\end{proof}



The next Lemma gives a full characterization of the Dirac graphs when $|S|=1$ and has an elementary proof. 

\begin{lemma}\label{all type b}
Let $H$ be a $3$-connected graph and suppose that $e:=uv\in E(H)$ is an edge of $H$ and $w\in V(H)$ is such that $\{u,v,w\}$ is a $3$-vertex cut of $H$. Then the following assertions are equivalent:
\begin{enumerate}
 \item [(a)] $H$ is $e$-Dirac.
 \item [(b)] Each connected component $K$ of $H\del \{u,v,w\}$ is a tree with a special vertex $x_K$ such that $N_H(w)\cap V(K)=\{x_K\}$.
\end{enumerate}
\end{lemma}

\section{One-, two- and no-sided separations}\label{sec-one-two}


Let $X$ be a $4$-vertex-cut of $G$ containing $\{u_1,u_2\}$. We say that $X$ is \defin{two-sided} if, for each component $K$ of $G\del X$, there is $\{i,j\}=\{1,2\}$ such that there is an edge from $u_i$ to $K$ but no edge from $u_j$ to $K$. We say that $X$ is \defin{one-sided} if, for each component $K$ of $G\del X$, there are edges from both $u_1$ and $u_2$ to $K$. We say that $X$ is \defin{no-sided} if $X$ is neither one-sided nor two-sided.

\begin{lemma}\label{ugly}
Suppose that $|S|=2$, $d_G(u_1),d_G(u_2)\ge4$ and all $4$-vertex-cuts of $G$ containing $\{u_1,u_2\}$ are no-sided. Then $G$ is obtained from a wheel by possibly doubling some spokes and splitting the hub by $e$.

\end{lemma}
\begin{proof}
\stepcounter{rotcount}Suppose that the lemma fails. First we check that $G\del X$ has exactly two connected components for each $4$-vertex cut $X$ containing $\{u_1,u_2\}$. Suppose that $X$ contradicts this assertion. As $X$ is no-sided, there is a connected component $H_1$ of $G\del X$ with neighbors of both $u_1$ and $u_2$ and two distinct connected components $H_2,H_3\neq H_1$ of $G\del X$. Then, for $X=\{u_1,u_2,x,y\}$, a cycle in the from $u_1,u_2,H_1,u_1$ avoids a cycle of the form $x,H_2,y,H_3,x$, a contradiction. Thus $G\del X$ has exactly two connected components.

Now, for each $2$-subset $A$ of $V(G)-\{u_1,u_2\}$ such that $A\cup\{u_1,u_2\}$ is a $4$-vertex cut of $G$, $A\cup\{u_1,u_2\}$ is no-sided by hypothesis. The connected component of $G\del A\cup\{u_1,u_2\}$ containg neighbors of both $u_1$ and $u_2$
 will be called the {\it $A$-right} component and denoted by $R(A)$, while the component containing neighbors of only one element of $\{u_1,u_2\}$ will be called the {\it $A$-left} component and denoted by $L(A)$. Choose $S$ maximizing the number of vertices in $L(S)$. Say that $u_1$ has a neighbor in $L(S)$ and $u_2$ does not.

Write $S=\{v_1,v_2\}$. By Lemma \ref{fan side}, $G[L(S)\cup\{u_1,v_1,v_2\}]$ is a fan with hub $u_1$ and endvertices $v_1$ and $v_2$. If $|R(S)|=1$, then it is clear that the lemma holds. So, $|R(S)|\ge 2$.

\begin{rot}\label{ugly-1}
For $i=1,2$, $v_i$ has at least two neighbors in $R(S)$. Moreover, $v_1u_2,v_2u_2,v_1v_2\notin E(G)$ and $u_2$ has at least $3$ neighbors in $R(S)$.
\end{rot}
\begin{rotproof}
For the first part, say that $v_1$ has an unique neighbor $x\in R(S)$. Then, for $A:=\{x,v_2\}$, $A\cup\{u_1,u_2\}$ separates $L(S)+ v_1$ from $R(S)\del x$, which is non-empty because $|R(S)|\ge 2$. If $u_2v_1\notin E(G)$, then there are no edges from $u_2$ to $L(S)\cup v_1$ and $L(A)=L(S)+v_1$, contradicting the maximality of $L(S)$. Thus  $u_2v_1\in E(G)$. So, both $u_1$ and $u_2$ have neighbors in $L(S)+v_1$. This implies that $R(A)=L(S)+v_1$ and $L(A)=R(S)\del x$. Now, for an unique $k\in\{1,2\}$, $u_k$ has neighbors in $L(A)$. By Lemma \ref{fan side}, $G[L(S)\cup S]$ is a $(v_1,v_2)$-path. By Lemma \ref{fan side} (for $S=A$), $G[L(A)\cup A]=G[R(S)\cup v_2]$ is a $(v_2,x)$-path. Recall that $x$ is the unique neighbor of $v_1$ in $R(S)$. Hence $G\del u_1,u_2$ is a cycle and the Lemma holds in this case, a contradiction. Therefore, $v_i$ has at least two neighbors in $R(S)$ for $i=1,2$.

Now, suppose for a contradiction that $v_2u_2\in E(G)$. Since there are two edges from $v_1$ to $R(S)$, hence $R(S)+v_1$ has a cycle, which must avoid a cycle of the form $u_1,u_2,v_2,L(S),u_1$, a contradiction. So, $v_2u_2\notin E(G)$. Similarly $v_1u_2\notin E(G)$. As $d_G(u_2)\ge 4$, $u_2$ has at least $3$ neighbors in $R(S)$. We already saw that $v_1v_2\notin E(G)$, since $G[L(S)\cup\{v_1,v_2\}]$ is a $(v_1,v_2)$-path. So, \ref{ugly-1} holds.
\end{rotproof}

\begin{rot}\label{ugly-0}
$G\del e$ is $3$-connected.
\end{rot}
\begin{rotproof}
Suppose for a contradiction that $P$ is a $2$-vertex cut of $G\del e$. As $G$ is $3$-connected, $P$ does not meet $\{u_1,u_2\}$ and $u_1$ and $u_2$ are in different connected components $K_1$ and $K_2$ of $G\del P\cup e$ respectively. If $|V(K_i)|=1$, then $d_G(u_i)\le 3$, contradicting our hypothesis, so $|V(K_i)|\ge 2$. This implies that $P\cup\{u_1,u_2\}$ is a $4$-vertex cut of $G$. But, since $G\del P\cup e$ has $u_1$ and $u_2$ in different connected components, it follows that $P\cup\{u_1,u_2\}$ is two-sided, contradicting our hypothesis.
\end{rotproof}

\begin{rot}\label{ugly-2}
$R(S)$ is a tree.
\end{rot}
\begin{rotproof} 
By \ref{ugly-0}, there are three pairwise disjoint $(\{v_1,v_2,u_1\},V(C))$-paths $\alpha_1$, $\alpha_2$ and $\alpha_3$ in $G\del e$. As $\{u_1,v_1,v_2\}$ separates $L(S)$ from $C$ in $G\del e$, none of these paths meet $L(S)$. Say that $\{v_1,a_1\}$, $\{v_2,a_2\}$ and $\{u_1,a_3\}$ are the respective pairs of endvertices of $\alpha_1$, $\alpha_2$ and $\alpha_3$ in $C$. Consider also $(u_2,V(C))$-paths $\beta_1$ and $\beta_2$ of $G\del u_1$ intersecting only in $u_2$ and let $b_1$ and $b_2$ be the respective endvertices of $\beta_1$ and $\beta_2$ in $C$. 

If for some $j\in\{1,2\}$ and $i\in \{1,2,3\}$, $\beta_j$ meets $\alpha_i$ out of $C$, then $G$ has a \\ $(u_2,\{v_1,v_2,u_1\})$-path $\gamma$ avoiding $C$ and, as a consequence, $C$ avoids a cycle of the form $u_1,u_2,\gamma,L(S)+\{v_1,v_2,u_1\}$ containing $e$, a contradiction. So, $\alpha_i$ and $\beta_j$ do not meet out of $C$.

Let $\delta$ and $\varepsilon$ be the the $(a_1,a_2)$-paths of $C$ meeting and avoiding $a_3$ respectively. If $\beta_j$ has an endvertex in $\int(\delta)$, then, the cycle $u_1,u_2,\beta_j,b_j,\delta,a_3,\alpha_3,u_1$ avoids a cycle of the form $v_1,\alpha_1,a_1,\varepsilon,a_2,v_2,L(S),v_1$, a contradiction. So, $b_1$ and $b_2$ are in $\varepsilon$ and we may assume that $b_1$ is closer to $a_1$ than $a_2$ in $\varepsilon$. See an illustration in Figure \ref{fig-ugly}.

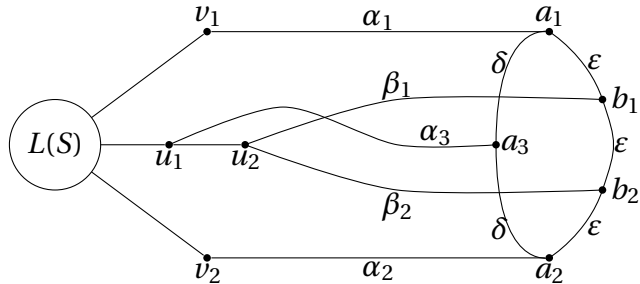
\begin{figure}[h]\label{fig-ugly}
\centering
\begin{tikzpicture}
\tikzstyle{node_style} =[shape = circle,fill = black,minimum size = 2pt,inner sep=1pt]
\node[node_style] (u1) at (-3,1.5) {};\node()at(-3.0,1.3) {$u_1$};
\node[node_style] (v1) at (-2.5,3.0) {};\node()at(-2.5,3.2) {$v_1$};
\node[node_style] (v2) at (-2.5,0.0) {};\node()at(-2.5,-0.2) {$v_2$};
\node[node_style] (u2) at (-2,1.5){};\node()at(-2,1.3) {$u_2$};
\node[node_style] (a1) at (2,3) {};\node()at(2,3.2) {$a_1$}; 
\node[node_style] (a2) at (2,0) {};\node()at(2,-0.2){$a_2$};
\node[node_style] (a3) at (1.3,1.5) {};\node() at (1.55,1.5) {$a_3$};
\node[node_style] (b1) at (2.7,2.1) {};\node() at (3,2.1) {$b_1$};
\node[node_style] (b2) at (2.7,0.9) {};\node() at (3,0.9) {$b_2$};
\draw plot [smooth, tension=1] coordinates {(a1) (b1) (b2) (a2)};
\draw plot [smooth, tension=2] coordinates {(a1) (a3) (a2)};
\draw (v1) to (a1); \node () at (-0.25,3.15) {$\alpha_1$};
\draw plot [smooth] coordinates {(u1)(-1.5,2)(0,1.5)(a3)}; \node () at (0.5,1.63) {$\alpha_3$};
\draw (v2) to (a2); \node () at (-0.25,-0.17) {$\alpha_2$};
\draw plot [smooth] coordinates{(u2)(0,2.1)(b1)};\node () at (0,2.29) {$\beta_1$};
\draw plot [smooth] coordinates{(u2)(0,0.9)(b2)};\node () at (0,0.68) {$\beta_2$};
\draw (u1)--(u2);
\node () at (2.95,1.5) {$\varepsilon$};
\node () at (2.6,2.6) {$\varepsilon$};
\node () at (2.6,0.4) {$\varepsilon$};
\node () at (1.35,2.6) {$\delta$};
\node () at (1.35,0.4) {$\delta$};
\node (los) at (-4.5,1.5) {};
\draw (v1)--(los);\draw (v2)--(los);\draw (u1)--(los);
\draw (los) [fill=white] circle (0.6cm);
\node () at (los) {$L(S)$};
\end{tikzpicture}
\caption{A illustration of the proof of \ref{ugly-2}}
\end{figure}

Consider the following subgraph of $G$ 
\begin{linenomath*}\[H:=
(L(S)+\{v_1,v_2,u_1\})\cup C\cup\alpha_1\cup\alpha_2\cup\alpha_3\cup\beta_1\cup\beta_2.\]
\end{linenomath*}
By \ref{ugly-1}, $v_1$ has two different neighbors in $R(S)$ and there is a $(v_1,V(H))$-path $\varphi$ of $R(S)+v_1$ beginning with an edge out of $H$. Let $x$ be the other endvertex of $\varphi$. Next we consider, for $i=1,2$, the following circuit:
\begin{linenomath*}\[D_i:=u_1,u_2,\beta_i,b_i,\varepsilon,a_i,\delta,a_3,\alpha_3,u_1.\]
\end{linenomath*}
Note that $V(D_1)\cap V(D_2)=V(\alpha_3)\cup u_2$ and that $H\del V(D_i)$ is connected for $i=1,2$. By Lemma \ref{starbucks}, $x\in V(\alpha_3)-u_1$. Now, $v_1,\alpha_1,a_1,\delta,a_3,\alpha_3,x,\varphi,v_1$ avoids a cycle of the form $u_1,u_2,\beta_2,\varepsilon,a_2,\alpha_2,v_2,L(S),u_1$, a contradiction.
\end{rotproof}

\begin{rot}\label{ugly-3}
$u_1$ and $u_2$ have no leaf of $R(S)$ as common neighbor.
\end{rot}
\begin{rotproof}
Suppose that $l$ is a leaf contradicting \ref{ugly-3}. By \ref{ugly-1}, $v_1$ and $v_2$ have neighbors in $R(S)-l$. Then, a cycle of the form $v_1,R(S)-l,v_2,L(S),v_1$ avoids the cycle $u_1,u_2,l,u_1$, a contradiction.
\end{rotproof}

\begin{rot}\label{ugly-4}
Let $l$ be a leaf of $R(S)$. Then, $u_2l\notin E(G)$, $u_1l\in E(G)$ and there is an unique index $i\in\{1,2\}$ such that $lv_i\in E(G)$.
\end{rot}
\begin{rotproof}
First let us prove that there is no leaf $l$ of $R(S)$ such that $lv_1,lv_2\in E(G)$. Suppose that $l$ is such a leaf. If $u_1$ has a neighbor in $V(R(S))-l$, then by \ref{ugly-3}, so does $u_2$ and, therefore, cycles of the form $u_1,u_2,R(S)\del l,u_1$ and $v_1,l,v_2,L(S),v_1$ avoid each other, a contradiction. Hence, $l$ is the unique neighbor of $u_1$ in $R(S)$. This implies that $l$ is the unique leaf of $R(S)$ incident to both $v_1$ and $v_2$. Let $l'$ be another leaf of $R(S)$. We may assume that $l'v_1,l'u_2\in E(G)$. 
As $l$ is the unique leaf incident to both $v_1$ and $v_2$, it follows from \ref{ugly-1} that $v_2$ has a neighbor in $R(S)\del l,l'$. So, $(R(S)\del l')+v_2$ has a cycle, which avoids a cycle of the form $u_1,u_2,l',v_1,L(S),u_1$, a contradiction. So, $v_1$ and $v_2$ have no leaf as common neighbor in $R(S)$.

By \ref{ugly-1}, this implies that each leaf of $R(S)$ is adjacent to an unique element of $\{u_1,u_2\}$ and an unique element of $\{v_1,v_2\}$. Let us prove that no leaf is adjacent to $u_2$. Suppose that some leaf $l$ is adjacent to $u_2$ and, say, $v_1$. By \ref{ugly-1}, $v_2$ has at least two neighbors in $R(S)\del l$ and cycles of the form $v_2,R(S)-l,v_2$ and $u_1,u_2,l,v_1,L(S),u_1$ avoid each other, a contradiction. This implies \ref{ugly-4}.
\end{rotproof}

\begin{rot}\label{ugly-5}
$R(S)$ has exactly two leaves.
\end{rot}
\begin{rotproof}
Suppose the contrary. Then, $R(S)$ has a vertex $w$ with $d_{R(S)}(w)\ge 3$. Moreover, $R(S)\del w$ has different connected components $K_1$, $K_2$ and $K_3$, each one containing a leaf of $R(S)$. Let $t\in\{1,2,3\}$. By \ref{ugly-4}, $K_t$ has a neighbor of $u_1$ and a neighbor of some $w_t\in\{v_1,v_2\}$. As $u_2$ has a neighbor in $R(S)-w$, we may assume that $u_2$ has a neighbor in $K_1$ and, therefore, $K_1+\{u_1,u_2\}$ has a cycle $C$ containing $e$. Now, there is a cycle of $(L(S)\cup K_2\cup K_3)+\{w,v_1,v_2\}$ avoiding $C$, a contradiction.
\end{rotproof}

We may write $R(S)$ as a path $w_1,\dots, w_n$. By \ref{ugly-1} and \ref{ugly-4} there are $1<a<b<n$ such that $w_a,w_b\in N_G(u_2)$. Also, by \ref{ugly-1} and \ref{ugly-4}, there are edges from $H:=L(S)+\{v_1,v_2\}$ to $w_1$, $w_n$ and a vertex $w_d$ with $1<d<n$. If $d<b$, then a cycle of the form $w_1,\dots,w_d,H,w_1$ avoids $u_1,u_2,w_b,\dots,w_n,u_1$, a contradiction. Thus $d\ge b$. But, symmetrically, we have $d\le a$, a contradiction.
\end{proof}

\section{Rope Bridges}\label{sec-rope bridges}
\renewcommand{\S}{\mathcal{S}}
In this section we describe structures that may appear when $\min\{d_G(u_1),d_G(u_2)\}=3$ or when $|S|=2$ and we may pick $S$ such that $S^+$ is two-sided. Those are the last characterizations that we need to prove Theorem \ref{main}; the theorem is proved in the end of the section as a direct consequence of Theorem \ref{main-cor}.

A graph $R$ with distinct vertices $u,x,y$ is an \defin{$(u,x,y)$-rope bridge} with \defin{ropes} $\rho_x$ and $\rho_y$ and \defin{family of steps} $\S$, if the following assertions hold.
\begin{enumerate}
\item [(RB0)] $d_R(v)\ge 3$ for all $v\in V(R)-\{u,x,y\}$.
\item [(RB1)] $\rho_x$ and $\rho_y$ are paths from $u$ to $x$ and $y$ respectively with $V(\rho_y)\cap V(\rho_x)=\{u\}$.
\item [(RB2)] $\S$ is a family of internally disjoint $(V(\rho_x)-u,V(\rho_y)-u)$-paths (whose members we call \defin{steps}). We denote by $x_\alpha$ and $y_\alpha$ the extremities of the step $\alpha$ in $\rho_x$ and $\rho_y$ respectively. We say that a step $\alpha$ \defin{crosses} a step $\beta$ if for some $\{s,t\}=\{x,y\}$, $s_\alpha$ and $s_\beta$ are distinct and appear in this order in $\rho_s$, while $t_\beta$ and $t_\alpha$ are distinct and appear in this order in $\rho_t$.
\item [(RB3)] Each step crosses at most one other step.
\item [(RB4)] Each vertex in $V(R)-V(\rho_y\cup\rho_x)$ is in some step.
\item [(RB5)] Each edge not in a member of $\S\cup\{\rho_x,\rho_y\}$ is incident to $u$.
\item [(RB6)] Let $z\in\{x,y\}$ and $v\in \rho_z$. Suppose that two steps have extremities in $\int(u,\rho_z,v)$. Then, $uv\notin E(R)$ and each step with extremity in $v$ has no inner vertices.
\end{enumerate}

\begin{lemma}\label{smaller rope-bridge}
Suppose that $R$ is an $(u,x,y)$-rope bridge with ropes $\rho_x$ and $\rho_y$ and that $w\in \int(\rho_y)$. Let $w_1$ be the vertex that follows $w$ in $\rho_y$, $\rho:=w_1,\rho_y,y$ and $\sigma_1,\dots,\sigma_k$ be the steps meeting $\rho$. Then, the graph $R'$ obtained from \\ $R\del (\rho\cup\int(\sigma_1),\dots,\int(\sigma_k))$ by suppressing the degree-two vertices in $\int(\rho_x)$ is an $(u,x,w)$-rope bridge.
\end{lemma}
\begin{proof}
Define $W=V(\rho\cup\int(\sigma_1)\cup\cdots\cup\int(\sigma_k))$. For some $F\cont E(\rho_x)$, up to isomorphisms, we have $R'=R\del W/F$. Consider $\rho'_x:=\rho_x/F$ and $\rho_w:=\rho_y\del V(\rho)$ as ropes for $R'$. Also consider as family of steps $\S'=\S-\{\sigma_1,\dots,\sigma_k\}$. Now, its is easy to verify that $R'$  with such ropes and steps, inherits each one of the properties (RB0)-(RB6) from $R$.
\end{proof}

\begin{lemma}\label{rope-bridge criterion}
Let $R$ be a connected graph with vertices $u,x,y$ and paths $\rho_x$ and $\rho_y$ from $u$ to $x$ and $y$ respectively, satisfying $V(\rho_y)\cap V(\rho_x)=\{u\}$. Suppose that $d_R(v)\ge 3$ for all $v\in V(R)-\{u,x,y\}$. Then, the following assertions are equivalent:
\begin{enumerate}
 \item [(a)] $R$ is an $(u,x,y)$-rope-bridge with ropes $\rho_x$ and $\rho_y$.
 \item [(b)] If $z\in \{x,y\}$ and $C$ is a cycle of $R\del \{u,z\}$, then $R$ has no $(u,z)$-path disjoint from $C$.
\end{enumerate}
\end{lemma}

\begin{proof}
Conditions (RB0) and (RB1) are given in the hypothesis. So, we have to prove that (b) is equivalent to (RB2)-(RB6) for some family of steps $\S$, which we will define ahead. Suppose that (b) holds.  

First we prove that each $v\in V(R)-(V(\rho_x)\cup(V\rho_y))$ is in a $(V(\rho_x),V(\rho_y))$-path $\sigma_v$ avoiding $u$. If $v$ is in a cycle of $R\del u$, then, by (b), this cycle must meet $\rho_x$ and $\rho_y$ and this implies the existence of $\sigma_v$. So, we may assume that $v$ is in no cycle of $R\del u$. As $d_{R\del u}(v)\ge 2$, this implies that $R\del u,v$ has different connected components $K_1$ and $K_2$ each one with an unique neighbor of $v$. If $\rho_x$ and $\rho_y$ are each one in a different component of $\{K_1,K_2\}$, then the existence of $\sigma_v$ is straightforward. So, we may assume that $K_1$ avoids $\rho_x$ and $\rho_y$. Let $w_1$ be the unique neighbor of $v$ in $K_1$. As $d_R(w_1)\ge 3$, $K_1$ has more than one vertex. But, for all $w\in V(K_1)-w_1$, $d_{R\del u,v}(w)\ge 2$. This implies that $K_1$ has a cycle, which must avoid $\rho_x$ and $\rho_y$, contradicting (b). This proves the existence of $\sigma_v$. 

Now, we define to be the elements of $\S$ the paths $\sigma_v$ with $v\in V(R)-V(\rho_x\cup\rho_y)$ and the paths of length one of the form $x',y'$ with $x'\in \rho_x-u$ and $y'\in \rho_y-u$. We will call {\it steps} the members of $\S$. For each step $\alpha$ we denote by $x_\alpha$ and $y_\alpha$ the endvertices of $\alpha$ in $\rho_x$ and $\rho_y$ respectively.

To establish (RB2) we shall prove that the steps are internally disjoint. Indeed, suppose for a contradiction that a step $\alpha$ intersects a step $\beta$ in an inner vertex $v$. As $\alpha\neq\beta$, we may assume that $x_{\alpha},\alpha,v\neq x_{\beta},\beta,v$. This implies that, $x_{\alpha},\alpha,v,\beta,x_{\beta},\rho_x,x_\alpha$ contains a cycle avoiding $\rho_y$, a contradiction. Thus, the steps are internally disjoint and (RB2) holds. 

Suppose that (RB3) does not hold. So, there is a step $\alpha$ crossing different steps $\beta$ and $\gamma$. We may assume that $x_\alpha$, $x_\beta$ and $x_\gamma$ appear in this order in $\rho_x$. Then, the cycle $x_\beta,\rho_x,x_\gamma,\gamma,y_\gamma,\rho_y,y_\beta,\beta,x_\beta$ avoids the path $u,\rho_x,x_\alpha,\alpha,y_\alpha,\rho_y,y$, a contradiction to (b). So, (RB3) holds. 

By construction, each vertex of $V(R)-V(\rho_x\cup\rho_y)$ is in a step and we have (RB4).

To prove (RB5) suppose that $f=vw$ is an edge of $R\del u$ not in $\rho_x\cup\rho_y$. If $\{v,w\}$ meet both $\rho_x$ and $\rho_y$ then $f$ is in a step. If $\{v,w\}\subseteq V(\rho_x)$, then there is a cycle in $\rho_x+f$ avoiding $\rho_y$. So, $\{v,w\}\ncont V(\rho_x)$ and, analogously, $\{v,w\}\ncont V(\rho_y)$. We may assume that $v\notin \rho_x\cup \rho_y$. By (RB4), $v$ is in the interior of a step $\sigma$. If $w$ is in $\sigma$ then either $f\in E(\sigma)$ or $\sigma+f$ has a cycle avoiding one of $\rho_x$ or $\rho _y$. If $w$ is in $\rho_x$ or $\rho_y$, say the former, then $w,v,\sigma,x_\sigma,\rho_x,w$ is a cycle avoiding $\rho_y$. Thus, $w$ is in the interior of a step $\alpha\neq \sigma$. Now $w,v,\sigma,x_\sigma,\rho_x,x_\alpha,\alpha,w$ is a cycle avoiding $\rho_y$, a contradiction. So, (RB5) holds.

Now, we prove (RB6). Let $z\in\{x,y\}$ and $v\in \rho_z$ and suppose that two steps $\alpha$ and $\beta$ have extremities in $\int(u,\rho_z,v)$. Say $z=x$. If $uv\in E(R)$, then the  path $u,v,\rho_x,x$ avoids the cycle $C:=x_\alpha,\alpha,y_\alpha,\rho_y,y_\beta,\beta,x_\beta,\rho_x,x_\alpha$, a contradiction. So, $uv\notin E(R)$. If a step $\sigma$ with extremity in $v$ has a inner vertex $w$, then, as $d_R(w)\ge 3$, there is an edge incident to $w$ not in $\sigma$, and by (RB5), this edge is $uw$. Now $u,w,\sigma,v,\rho_x,x$ avoids $C$, a contradiction. So, (RB6) holds and (b) implies (a).

Suppose that $R$ is a graph for which (a) holds but (b) does not hold. Choose $R$ with $|V(R)|$ as small as possible. Consider $z\in \{x,y\}$ such that there is a cycle $C$ of $R\del \{u,z\}$ and a $(u,z)$-path $\gamma$ disjoint from $C$ as short as possible. Say that $z=x$. Let $v$ be the vertex of $\gamma$ such that $v,\gamma,x$ is contained in $\rho_x$ and $v,\gamma,x$ is as long as possible. Let $w$ be the vertex preceding $v$ in $\gamma$. Note that $C$ meets at least two steps, whose endvertices precedes $v$ in $\rho_x$. By (RB6), $w\neq u$ and all steps arriving in $v$ have no internal vertices.  If the edge $wv$ is in $\rho_x$, then $w,v,\rho_x,x$ violates the maximality of $v,\rho_x,x$. So, as $w\neq u$, $wv$ is in a step, which may not contain internal vertices. Hence $w\in V(\rho_y)-u$ and $v,w$ is a step. If $w=y$, then $u,\gamma,w$ violates the minimalty of $\gamma$ (for $z=y$), a contradiction. Let $t$ be the vertex following $w$ in $\rho_y$. If $C$ meets $\eta:=t,\rho_y,y$, then, at least two steps contained in $C$ have endvertices in $\eta$, but these steps also have endvertices in $\int(u,\rho_x,v)$ and, therefore, cross $v,w$, contradicting (RB3). So, $C$ does not meet $\eta$. Let $R'$ be obtained from $R$ by deleting $V(\eta)$ and all inner vertices of steps with endpoints in $\eta$ and, then, suppressing the degree-$2$ vertices. By Lemma \ref{smaller rope-bridge}, $R'$ is an $(u,x,w)$-rope bridge with less vertices than $R$. But, the path induced by $u,\gamma,w$ and the cycle induced by $C$ in $R'$ contradict (b). This is a contradiction to the minimality of $|V(R)|$.
\end{proof}

We say that a step in a rope-bridge is \defin{short} if it has no inner vertices and \defin{long} otherwise. A vertex $v$ in a rope $rho$ of an $(u,x,y)$-rope bridge is \defin{clean} if $uv\notin E(G)$ provided $v$ is not the vertex following $u$ in $\rho$ and all steps arriving in $v$ are short. Note that (RB6) says that for the rope $\rho$ containing $v$, $v$ is clean if two steps arrive in $\int(u,\rho,v)$.

The following lemma follows directly from (RB0)-(RB6)

\begin{lemma}\label{constructions}
Let $R$ be an $(u,x,y)$-rope bridge with ropes $\rho_x$ and $\rho_y$. Write, for $z\in\{x,y\}$, $z_0:=u$ and $\rho_z:=z_0,z_1,\dots,z_{n(z)}$.
\begin{enumerate}

 \item [(a)] Suppose that, for $m\ge 2$ and $t\ge 1$, $\alpha_1,\dots,\alpha_m$ are the steps arriving in $x_t$. Let $y_{k(i)}:=y_{\alpha_i}$ and suppose $k(1)\le\cdots\le k(m)$. Let $R_1$ be the graph obtained from $R$ by splitting $x_t$ into vertices $v$ and $w$ in such a way that:
 \begin{itemize}
  \item $x_{t-1}v\in E(R_1)$,
  \item either $x_t\neq x$ and $x_{t+1}w\in E(R_1)$ or $x_t=x$ and we consider $w=x$,
  \item for some $1\le l<m$, the paths corresponding to $\alpha_1,\dots,\alpha_l$ arrive in $v$ and the ones corresponding to $\alpha_{l+1},\dots,\alpha_m$ arrive in $w$ in $R_1$ and
  \item $uw\notin E(R_1)$ and $uv\in E(R_1)$ if and only if $ux_t\in E(R_1)$.
 \end{itemize}
Suppose that all steps arriving in $x_{t+1},\dots,x_{n(x)}$ are short. Then, $R_1$ is an $(u,x,y)$-rope bridge if one of the following assertions hold:
 \begin{enumerate}
  \item [(a1)] $\alpha_{l+1},\dots,\alpha_m$ are short, or
  \item [(a2)] $t=l=1$.
 \end{enumerate}
 Moreover, a similar construction with $x$ and $y$ playing swapped roles also results in an $(u,x,y)$ rope-bridge.
\item [(b)] Suppose that $\alpha$ is a step with endvertices in $x_a$ and $y_b$, $x_{a},\dots,x_{n(x)},y_{b},\dots, y_{n(y)}$ are all clean, $\alpha$ crosses no other step, and no other step has an endvertex in $x_a$ or $y_b$. Let $R_2$ be the graph obtained from $R$ by deleting the edge of $\alpha$, splitting $x_a$ into vertices $v_1$ and $v_2$ and $y_b$ into $w_1$ and $w_2$, then, adding the edges $v_1w_2$ and $v_2w_1$ as steps in such a way that:
\begin{itemize}
 \item $x_{a-1}v_1,y_{b-1}w_1\in E(R_2)$,
 \item either $x_a\neq x$ and $x_{a+1}v_2\in E(R_2)$ or $x_a=x$ and we consider $v_2=x$ and
 \item either $y_b\neq y$ and $y_{b+1}w_2\in E(R_2)$ or $y_b=y$ and we consider $w_2=y$.
\end{itemize}
Then $R_2$ is an $(u,x,y)$-rope bridge. 
\end{enumerate}
\end{lemma}


\begin{lemma}\label{rope bridge is a minor}
If $R$ is an $(u,x,y)$-rope bridge then, up to the labels of elements other than $u$, $x$ and $y$, $R$ is a minor of a graph as in Figures \ref{fig-side1} or \ref{fig-side2} (paths that contracts to the ropes are drawn in thick lines).
\end{lemma}
\begin{center}
 \begin{figure}[h]
  \begin{minipage}{6cm}
   \begin{tikzpicture}\centering[scale=3]
     \tikzstyle{node_style2} = [shape = circle,fill = black,minimum size = 0.5pt,inner sep=0.5pt]
      \tikzstyle{node_style} =[shape = circle,fill = black,minimum size = 2pt,inner sep=1pt]
      \begin{scope}
	\fill[gray!20,opacity=1] (2,0) -- (0.5,-2) -- (1.25,-2) -- cycle;
	\fill[gray!20,opacity=1] (2,0) -- (1.75,-2) -- (2.5,-2) -- cycle;
	\fill[gray!20,opacity=1] (2,0) -- (2.75,-2) -- (3.5,-2) -- cycle;
	\fill[gray!20,opacity=1] (2,0) -- (4,-1) -- (4,1);
	
	\node[node_style] (v) at (2,0) {};\node () at (1.8,0) {$u$};
	\node[node_style] (x1) at (4,1) {};\node[node_style] (x2) at (5,1) {};
	\node[node_style] (y1) at (4,-1) {};\node[node_style] (y2) at (5,-1) {};
	
	\node[node_style2] () at (0.9,-2) {};
	\node[node_style2] () at (1,-2) {};
	\node[node_style2] () at (1.1,-2) {};
	\node[node_style] (lfl) at (0.5,-2)  {};\node[node_style] (lfm) at (0.75,-2) {}; \node[node_style] (lfr) at (1.25,-2) {};
	\node[node_style2] () at (1.4,-2){};
	\node[node_style2] () at (1.5,-2){};
	\node[node_style2] () at (1.6,-2){};
	\node[node_style2] () at (2.15,-2){};
	\node[node_style2] () at (2.25,-2){};
	\node[node_style2] () at (2.35,-2){};
	\node[node_style] (mfl) at (1.75,-2) {};\node[node_style] (mfm) at (2,-2) {}; \node[node_style] (mfr) at (2.5,-2) {};
	\node[node_style2] () at (3.15,-2){};
	\node[node_style2] () at (3.25,-2){};
	\node[node_style2] () at (3.35,-2){};
	\node[node_style] (rfl) at (2.75,-2) {};\node[node_style] (rfm) at (3,-2) {}; \node[node_style] (rfr) at (3.5,-2) {};
	
	\draw (y2)--(y1)--(v)--(x1)--(x2);
	\draw (v)--(lfl)--(lfm)--(v)--(lfr);
	\draw (v)--(mfl)--(mfm)--(v)--(mfr);
	
	\draw (v)--(rfl)--(rfm)--(v)--(rfr);
	
	\node[node_style] (1) at (4,-0.6){};
	\node[node_style] (2) at (4,-0.2){};
	\node[node_style] (3) at (4,0.6){};
	\node[node_style2] () at (4,0.1){};
	\node[node_style2] () at (4,0.2){};
	\node[node_style2] () at (4,0.3){};
	\draw (y1)--(1)--(v)--(2)--(1)--(v)--(3)--(x1);	
	
	\draw (y2)--(rfr);
	\draw (y2)--(3.7,-2.1)--(2.5,-2.1)--(2.5,-2);
	\draw (y2)--(3.9,-2.2)--(1.25,-2.2)--(1.25,-2);
	\draw (x2)--(5.2,0)--(5.2,-2.4)--(2.75,-2.4)--(rfl);
	\draw (x2)--(5.3,0)--(5.3,-2.5)--(1.75,-2.5)--(mfl);
	\draw (x2)--(5.4,0)--(5.4,-2.6)--(0.5,-2.6)--(lfl);
	
	\node[node_style] (x3) at (5,1.5){};\node[node_style] (x4) at (4.5,1.5){};
	\node[node_style] (y3) at (5,2){};\node[node_style] (y4) at (4.5,2){};
	\draw (x3)--(y4);\draw(y3)--(x4);
	\node[node_style] (x5) at (4,1.5){};\node[node_style] (x6) at (3.5,1.5){};
	\node[node_style] (y5) at (4,2){};\node[node_style] (y6) at (3.5,2){};
	\draw (x5)--(y6);\draw(y5)--(x6);
	\node[node_style] (x7) at (2.5,1.5){};\node[node_style] (x8) at (2,1.5){};
	\node[node_style] (y7) at (2.5,2){};\node[node_style] (y8) at (2,2){};
	\draw (x7)--(y8);\draw(y7)--(x8);
	\node[node_style](x9)at(1.5,1.5){};\node[node_style](x10)at(1,1.5){};\node()at(0.8,1.5){$x$};
	\node[node_style] (y9) at (1.5,2){};\node[node_style] (y10) at (1,2){};\node()at(0.8,2){$y$};
	\draw (x9)--(y10);\draw(y9)--(x10);
	
	\draw [line width=1](v)--(x1)--(x2)--(x3)--(3.3,1.5);
	\node () at (3.0,1.45) {$\cdots$};
	\draw [line width=1] (2.7,1.5)--(x10);
	\draw [line width=1] (v)--(y1)--(y2)--(5.6,-1)--(5.6,2)--(3.3,2);
	\node () at (3.0,1.95) {$\cdots$};
	\draw [line width=1](2.7,2)--(y10);
\end{scope}
\end{tikzpicture}
\caption{}\label{fig-side1}
\end{minipage}
\begin{minipage}{6cm}
\begin{tikzpicture}\centering
\tikzstyle{node_style2} = [shape = circle,fill = black,minimum size = 0.5pt,inner sep=0.5pt]
\tikzstyle{node_style} =[shape = circle,fill = black,minimum size = 2pt,inner sep=1pt]
\begin{scope}
	\fill[gray!20,opacity=1] (2,0) -- (0.5,-2) -- (1.25,-2) -- cycle;
	\fill[gray!20,opacity=1] (2,0) -- (1.75,-2) -- (2.5,-2) -- cycle;
	\fill[gray!20,opacity=1] (2,0) -- (2.75,-2) -- (3.5,-2) -- cycle;
	\fill[gray!20,opacity=1] (2,0) -- (4,-1) -- (4.4,-0.2);
	\fill[gray!20,opacity=1] (2,0) -- (4,1) -- (4.4, 0.2);
	
	\node[node_style] (v) at (2,0) {};\node () at (1.8,0) {$u$};
	\node[node_style] (x1) at (4,1) {};\node[node_style] (x2) at (5,1) {};
	\node[node_style] (y1) at (4,-1) {};\node[node_style] (y2) at (5,-1) {};
	
	\node[node_style2] () at (0.9,-2) {};
	\node[node_style2] () at (1,-2) {};
	\node[node_style2] () at (1.1,-2) {};
	\node[node_style] (lfl) at (0.5,-2)  {};\node[node_style] (lfm) at (0.75,-2) {}; \node[node_style] (lfr) at (1.25,-2) {};
	\node[node_style2] () at (1.4,-2){};
	\node[node_style2] () at (1.5,-2){};
	\node[node_style2] () at (1.6,-2){};
	\node[node_style2] () at (2.15,-2){};
	\node[node_style2] () at (2.25,-2){};
	\node[node_style2] () at (2.35,-2){};
	\node[node_style] (mfl) at (1.75,-2) {};\node[node_style] (mfm) at (2,-2) {}; \node[node_style] (mfr) at (2.5,-2) {};
	\node[node_style2] () at (3.15,-2){};
	\node[node_style2] () at (3.25,-2){};
	\node[node_style2] () at (3.35,-2){};
	\node[node_style] (rfl) at (2.75,-2) {};\node[node_style] (rfm) at (3,-2) {}; \node[node_style] (rfr) at (3.5,-2) {};
	
	\draw (y2)--(y1)--(v)--(x1)--(x2);
	\draw (v)--(lfl)--(lfm)--(v)--(lfr);
	\draw (v)--(mfl)--(mfm)--(v)--(mfr);
	
	\draw (v)--(rfl)--(rfm)--(v)--(rfr);
	
	\node[node_style] (lofl) at (4.1,-0.8){};\node[node_style] (lofu) at (4.4,-0.2){}; 	\draw (y1)--(lofl)--(v)--(lofu);
	\node[node_style2] (ldot1) at (4.2,-0.6){};
	\node[node_style2] (ldot2) at (4.25,-0.5){};
	\node[node_style2] (ldot3) at (4.3,-0.4){};
	\node[node_style] (upfu) at (4.1,0.8){};\node[node_style] (upfl) at (4.4, 0.2){};	\draw (x1)--(upfu)--(v)--(upfl);
	\node[node_style2] (ldot1) at (4.2,0.6){};
	\node[node_style2] (ldot2) at (4.25,0.5){};
	\node[node_style2] (ldot3) at (4.3,0.4){};
	\draw (lofu)--(x2);
	\draw (upfl)--(y2);
	
	\draw (y2)--(rfr);
	\draw (y2)--(3.7,-2.1)--(2.5,-2.1)--(2.5,-2);
	\draw (y2)--(3.9,-2.2)--(1.25,-2.2)--(1.25,-2);
	\draw (x2)--(5.2,0)--(5.2,-2.4)--(2.75,-2.4)--(rfl);
	\draw (x2)--(5.3,0)--(5.3,-2.5)--(1.75,-2.5)--(mfl);
	\draw (x2)--(5.4,0)--(5.4,-2.6)--(0.5,-2.6)--(lfl);
	
	\node[node_style] (x3) at (5,1.5){};\node[node_style] (x4) at (4.5,1.5){};
	\node[node_style] (y3) at (5,2){};\node[node_style] (y4) at (4.5,2){};
	\draw (x3)--(y4);\draw(y3)--(x4);
	\node[node_style] (x5) at (4,1.5){};\node[node_style] (x6) at (3.5,1.5){};
	\node[node_style] (y5) at (4,2){};\node[node_style] (y6) at (3.5,2){};
	\draw (x5)--(y6);\draw(y5)--(x6);
	\node[node_style] (x7) at (2.5,1.5){};\node[node_style] (x8) at (2,1.5){};
	\node[node_style] (y7) at (2.5,2){};\node[node_style] (y8) at (2,2){};
	\draw (x7)--(y8);\draw(y7)--(x8);
	\node[node_style](x9)at(1.5,1.5){};\node[node_style](x10)at(1,1.5){};\node()at(0.8,1.5){$x$};
	\node[node_style] (y9) at (1.5,2){};\node[node_style] (y10) at (1,2){};\node()at(0.8,2){$y$};
	\draw (x9)--(y10);\draw(y9)--(x10);
	
	\draw [line width=1](v)--(x1)--(x2)--(x3)--(3.3,1.5);
	\node () at (3.0,1.45) {$\cdots$};
	\draw [line width=1] (2.7,1.5)--(x10);
	\draw [line width=1] (v)--(y1)--(y2)--(5.6,-1)--(5.6,2)--(3.3,2);
	\node () at (3.0,1.95) {$\cdots$};
	\draw [line width=1](2.7,2)--(y10);
\end{scope}
\end{tikzpicture}
\caption{}
\label{fig-side2}
\end{minipage}
\end{figure}
\end{center}
\begin{proof}
Suppose that $R'$ is a graph contradicting the lemma. Let $R$ be a graph obtained from $R'$ using the operations of Lemma \ref{constructions} up to the point that they can no longer be performed. If we prove the lemma for $R$, it will also hold for $R'$. If all vertices in the ropes are clean the result is clear. So, we may assume that there is a non-clean vertex. Write $\rho_z=u,z_1,\dots,z_{n(z)}$ for each $z\in \{x,y\}$. Consider the smallest index $c(z)$ such that all vertices $z_{c(z)+1},\dots,z_{n(z)}$ are clean. By (RB6), $1\le c(z) \le 2$. 

If $c(z)<t\le n(z)$, there is an unique step arriving in $z_t$ since operation (a) of Lemma \ref{constructions} cannot be performed and (a1) would hold otherwise. In particular, the step arriving in $z_t$ must cross another step since operation (b) of Lemma \ref{constructions} also cannot be performed. 

Also, there is an unique step $\alpha_z$ arriving in $z_1$, as we prove next. If $c(z)=2$, it follows from (RB6). If $c(z)=1$, it follows from the fact that item (a1) of Lemma \ref{constructions} does not hold.

If there are no long steps, $x_2$ and $y_2$ are the unique possibly non-clean vetices and the result holds. So, we may assume that there is some long step.

If $c(x)=1$, then $x_1$ is the unique non-clean vertex of $\rho_x$ and, as an unique step arriving at $x_1$, $\alpha_x$ is the unique long step. If $\alpha_y=\alpha_x$, $R$ is a minor of a graph as in Figure \ref{fig-side1}, otherwise it is a minor a graph as in Figure \ref{fig-side2}. 

So, we may assume that $c(x)=2$ and, analogously, that $c(y)=2$. If $\alpha_x=\alpha_y$, then $R$ is a minor of a graph as in Figure \ref{fig-side1}. Otherwise, as each step cross at most one other, $\alpha_x$ has an endvertex in $y_2$ and $\alpha_y$ has an endvertex in $x_2$. It follows that $R$ is a minor of a graph as in Figure \ref{fig-side2}.
\end{proof}


\begin{lemma}\label{components are rope-bridges}
Suppose that $S=\{x,y\}$ and let $\{i,j\}=\{1,2\}$. Moreover, suppose
\begin{enumerate}
 \item [(i)] $d_G(u_i)\ge 4$ and $S^+$ is a two-sided vertex-cut or
 \item [(ii)] $N_G(u_i)=\{u_j,x,y\}$.
\end{enumerate}
For $\{k,l\}=\{1,2\}$, we define an \defin{$u_k$-component} as a connected component of $G\del S^+$ with a neighbor of $u_k$ and no neighbor of $u_l$. Then, one of the following assertions hold:
\begin{enumerate}
 \item [(a)] $N_G(u_j)=\{x,y,u_i\}$,
 \item [(b)] There is an unique $u_j$-component $K$ and $K+\{u_j,x,y\}$ is a $(u_j,x,y)$-rope bridge.
 \item [(c)] There is at least two $u_j$-components and, for each $u_j$-component $K$, $K+\{u_j,x,y\}$ is a fan with hub $u_j$ and endvertices $x$ and $y$ or a wheel with hub $u_j$ and $xy$ as edge.
\end{enumerate}
\end{lemma}
\begin{proof}
Suppose that (a) fails. Then, there is some $u_j$-component. Note that either there is an $u_i$-component or $N_G(u_i)\cont S^+$. In both cases, for each $z\in\{x,y\}$, there is a $(z,u_i)$-path $\beta_z$ whose internal vertices are all in an $u_i$-component (in case they exist).

First suppose that there is an unique $u_j$-component $K$. By Menger's Theorem, there are three internally disjoint paths from $u_j$ to $u_i$, one of these must contain $x$, namely $\zeta_x$, and other must contain $y$, call it $\zeta_y$. We define, for $z=x,y$, $\rho_z:=u_j,\zeta_z,z$. By the uniqueness of $K$, $\rho_x$ and $\rho_y$ are in $R:=K+\{u_j,x,y\}$. We shall prove that $R$ is an $(u_j,x,y)$-rope-bridge with ropes $\rho_x$ and $\rho_y$. For this we use Lemma \ref{rope-bridge criterion}; if $R$ does not satisfy item (b) of that lemma, there is $z\in\{x,y\}$, a cycle $C$ of $R\del\{u_j,z\}$ and an $(u_j,z)$-path $\alpha$ avoiding $C$. Now $u_i,u_j,\alpha,z,\beta_z,u_i$ is a cycle avoiding $C$, a contradiction.  Thus, (b) holds.

Now suppose that there are $u_j$-components $K_1,\dots,K_m$ for some $m\ge 2$. Let us show that each $K_k$ is a tree. Suppose for a contradiction that $D$ is a cycle in, say, $K_1$.  Now a cycle of the form $u_i,u_j,K_2,x,\beta_x,u_i$ avoids $D$, a contradiction. Then $K_k$ is a tree for $k=1,\dots,m$.

Let us prove that, for each $k=1,\dots,m$, each element of $\{x,y\}$ has an unique neighbor in $K_k$. Say that $y$ have two different neighbors in $K_1$. Now, $K_1+y$ has a cycle $C$. Note that $C$ avoids a cycle of the form $u_i,u_j,K_2,x,\beta_x,u_i$, a contradiction. So, each element of $\{x,y\}$ has an unique neighbor in $K_k$ for each $k=1,\dots,m$.

If $K_k$ has three different leaves, then two of them have a same neighbor in $\{x,y\}$ a contradiction. Then $K_k$ is a path. Therefore, each leaf of $K_k$ has a different neighbor in $\{x,y\}$ and each vertex of $K_k$ must have $u_j$ as neighbor. This implies (c). 
\end{proof}

Recall the definitions of type (a), (b) and (c) graphs from Theorem \ref{main}. It is straighforward to check that all graphs decribed in items (a), (b) and (c) of Theorem \ref{main} have no pair of disjoint cycles whose union contains $e$. Therefore, Theorem \ref{main} is a direct consequence of the following theorem:

\begin{theorem}
\label{main-cor}
Suppose that $G$ is a $3$-connected graph with at least six vertices and an edge $e:=u_1u_2$ such that there is no pair of disjoint cycles of $G$ which union contains the edge $e$. Let $S$ be smallest set such that $S^+:=S\cup\{u_1,u_2\}$ contains a vertex cut of $G$. Then one of the following assertions holds.
\begin{enumerate}
\item [(a)] $|S|=2$, $d_G(u_1),d_G(u_2)\ge 4$, we may not pick $S$ such that $S^+$ is two-sided and $G$ is a type (a) $e$-Dirac graph.
\item [(b)] $|S|=1$ and $G$ is a type (b) $e$-Dirac graph.
\item [(c)] $|S|=2$, $G$ is a type (c) $e$-Dirac graph, we may pick $S$ with the property that for some $\{i,j\}=\{1,2\}$ either $N_G(u_i)=S\cup u_2$ or $d_G(u_i)\ge 4$ and $S^+$ is a $2$-sided vertex cut. Moreover, we call the {\bf $u_j$-components} the connected components of $G-S^+$ with a neighbor of $u_j$ and one of the following assertions holds:
\begin{enumerate}
 \item [(c1)] $N_G(u_j)=S\cup u_i$;
 \item [(c2)] there is only one $u_j$-component $J$, which has the property that, for $S=\{x,y\}$, $J+\{u_j,x,y\}$ is an $(u_j,x,y)$-rope bridge; or
 \item [(c3)] there is more that one $u_j$-component and each $u_j$-conponent $J$ has the property that either $J+(S\cup u_j)$ either is a fan with $u_j$ as hub and with endvertices in $S$ or is a wheel with $u_2$ as hub and contains an edge linking the vertices of $S$.
\end{enumerate}
\end{enumerate}
\end{theorem}

\begin{proof}
Suppose that the theorem fails. If $|S|=1$, (b) follows from Lemma \ref{all type b}. Assume $|S|\ge2$. We will split the proof into two cases as follows.
\smallskip

{\it Case 1: Some vertex in $\{u_1,u_2\}$ has degree $3$ in $G$ or we may pick $S$ such that $S^+$ is a two-sided vertex-cut. }

By Lemma \ref{Sle2}, the hypothesis of Lemma \ref{components are rope-bridges} holds. Consider the terminologies as in that lemma. Let $\{i,j\}=\{1,2\}$ and let $K_1,\dots,K_m$ be the $u_j$-components. To prove that $G$ is a type (c) graph it suffices to prove that there is an isomorphism $\Phi$ between $H:=G[K_1\cup\cdots\cup K_m\cup \{u_j,x,y\}]$ and a minor of a graph like in Figures \ref{fig-side1} or \ref{fig-side2} such that $\Phi$ preserves $x$ and $y$ and $\varphi(u_j)=u$. The result is clear if items (a) or (c) of Lemma \ref{components are rope-bridges} holds. So we may assume that item (b) of that lemma holds. Now, (c) follows from Lemma \ref{rope bridge is a minor}.

\smallskip

{\it Case 2: $d_G(u_1),d_G(u_2)\ge 4$ and we cannot choose $S$ so that $S^+$ is a two-sided vertex cut.}

As $d_G(u_1),d_G(u_2)\ge 4$, by Lemma \ref{Sle2}, $|S|=2$. If all $4$-vertex cuts of $G$ are no-sided, (a) follows from Lemma \ref{ugly}. Assume the contrary. As Case 1 does not hold, we may pick $S=\{x,y\}$ in such a way that $S^+$ is one-sided. 

If $\kappa\ge 3$, then cycles of the form $u_1,u_2,G_1,u_1$ and $x,G_2,y,G_3,x$ avoid each other. Thus, $\kappa=2$. If an element of $S$ has two edges to a same component of $G\del S^+$, say $x$ has two edges to $G_1$, then cycles of the form $x,G_1,x$ and $u_1,u_2,G_2,u_1$ avoid each other. Thus, each element of $S$ has exactly one edge to each component of $G\del S^+$. 

Let $\{i,j\}=\{1,2\}$. If $G_i$ has a cycle, then it avoids a cycle of the form $u_1,u_2,G_j,u_1$. Hence $G_i$ is a tree. If $l$ is a leaf of $G_i$ with no neighbor in $S$, then $u_1l,u_2l\in E(G)$ and $S$ has neighbors in $G_i\del l$. This implies that a cycle of $G\del{u_i,u_j,l}$ avoids $u_1,u_2,l,u_1$, a contradiction. Thus each leaf of $G_i$ has a neighbor in $\{x,y\}$. As a consequence, if $G_i$ has three different leaves, then two of them have a same neighbor in $S$ and a vertex of $S$ has two different neighbors in $G_i$, a contradiction. Therefore, we may write $G_i$ as a path $w_1,\dots,w_n$ with $N_G(x)\cap V(G_i)=\{w_1\}$ and $N_G(y)\cap V(G_i)=\{w_n\}$. Now, to prove (a), it is left to check that $xy\notin E(G)$. Indeed, suppose the contrary. Then, $G_i+\{x,y\}$ has a cycle $C$. But, as $S^+$ is two sided, $G_2+\{u_1,u_2\}$ has a cycle containing $e$ and avoiding $C$, a contradiction.
\end{proof}

\section{Strongly $e$-Dirac graphs}\label{sec-strong}
In this section, we prove Theorems \ref{main-strong} and \ref{main-strong-2con}. We say that a graph $G$ with an edge $e$ is \defin{strongly $e$-Dirac} if $G$ has no pair of edge-disjoint cycles whose union contains $e$. Clearly all strongly $e$-Dirac graphs are $e$-Dirac graphs.

\begin{lemma}\label{str type b}
All type (b) $e$-Dirac $3$-connected graphs with more than five vertices are not strongly $e$-Dirac.
\end{lemma}
\begin{proof}
Suppose that $G$ is a graph contradicting the lemma and let $e=uv$. Consider a $3$-vertex-cut $S:=\{u,v,w\}$ of $G$ as in item (b) of Theorem \ref{main}. If $G\del S$ has three distinct connected components $K_1$, $K_2$ and $K_3$, then $G$ has edge-disjoint cycles of the form $u,v,K_1,u$ and $v,K_2,w,K_3,v$, a contradiction. Thus $G\del S$ has exactly two connected components, $K$ and $K'$. As $|G|\ge 6$, we may assume that $|K|\ge 2$. Note that there are two leaves $l_1$ and $l_2$ in $K$ such that $x_K\neq l_1$. As $d_G(l_2)\ge 3$, $l_2$ has a neighbor in $\{u,v\}$, say $v$, while both $u$ and $v$ are adjacent to $l_1$. Consider the $(l_1,l_2)$-path $\gamma$ of $K$. Now, $G$ has $v,l_1,\gamma,l_2,v$ edge-disjoint from a cycle of the form $u,K',v,u$, a contradiction.
\end{proof}

\begin{lemma}\label{str type a}
If $G$ is a $3$-connected type (a) strongly $e$-Dirac graph, then $G$ is isomorphic to the prism or to $K_{3,3}$.
\end{lemma}
\begin{proof}
Suppose that $G$ contradicts the lemma and let $G$ be obtained by splitting the hub of a wheel with rim $R:=x_1,\dots,x_n$. If, in $G$,  $x$ is a common neighbor to $u$ and $v$, then $u,v,x,u$ is disjoint from the rim, a contradiction, therefore $u$ and $v$ has no common neighbor in $G$. This implies that $n\ge 4$. Let us prove that $n=4$. Suppose the contrary. Then we may assume that $u$ has three neighbors $y_1$, $y_2$ and $y_3$ in  this order in a cyclic ordering of $R$ and $v$ has a neighbor $z$ in $R$, say after $y_3$ and before $y_1$ in this same ordering. Let $\alpha$ be the $(y_1,y_2)$-path of $R$ avoiding $z$ and $\beta$ the $(y_3,z)$-path of $R$ avoiding $\alpha$. Now $u,y_1,\alpha,y_2,u$ is edge-disjoint from $v,u,y_3,\beta,z,v$, a contradiction. So, $n=4$. Now the result is clear.
\end{proof}

We say that an $(u,x,y)$-rope-bridge is \defin{strong} if all its steps have length one and no pair of steps have a common endvertex.

\begin{lemma}\label{str-char1}
A graph $R$ with a $3$-subset $\{u,x,y\}\cont V(G)$ is a strong $(u,x,y)$-rope-bridge with ropes $\rho_x$ and $\rho_y$ if and only if 
$R$ has internally disjoint paths $\rho_x:=v_0,\dots,v_n$ and $\rho_y:=w_0,\dots,w_n$, with $v_n=x$, $w_n=y$ and $v_0=w_0=u$ and there is a family $\mathcal{P}$ of pairwise disjoint pairs of consecutive elements of $\{1,\dots,n\}$ such that 
\begin{enumerate}
\item [(a)] $V(R)=V(\rho_x)\cup V(\rho_y)$ and
\item [(b)] $E(R)=\{v_aw_b,v_bw_a:\{a,b\}\in \mathcal P\}\cup\{v_cw_c:c$ is in no member of $\mathcal P\}\cup E(\rho_x)\cup E(\rho_y)$.
\end{enumerate}
\end{lemma}
\begin{proof}
It is clear that a graph satisfying the given conditions is a strong $(u,x,y)$-rope-bridge with ropes $\rho_x$ and $\rho_y$. Let us prove the converse. As no pair of steps shares the same endvertices, then $\rho_x$ and $\rho_y$ have the same number of vertices and we can label $\rho_x=v_0,\dots,v_n$ and $\rho_y=w_0,\dots,w_n$. We let $\sigma_i$ be the step with extremity in $v_i$. Let $\mathcal P$ be the family of pairs $\{i,j\}$ such that $\sigma_i$ crosses $\sigma_j$. By (RB3), the members of $\mathcal P $ are pairwise disjoint. It also follows from (RB3) that each pair in $\mathcal P$ contains consecutive indices. Analogously, for a pair $\{i,j\}\in \mathcal P$, $\sigma_i$ and $\sigma_j$ also have endvertices that are neighbors in $\rho_y$. This implies the lemma.
\end{proof}

\begin{lemma}\label{strong rope-bridge criterion}
Let $R$ be a connected graph with vertices $u,x,y$ and paths $\rho_x$ and $\rho_y$ from $u$ to $x$ and $y$ respectively, satisfying $V(\rho_y)\cap V(\rho_x)=\{u\}$. Suppose that $d_G(v)\ge 3$ for all $v\in V(R)-\{u,x,y\}$. Then the following assertions are equivalent:
\begin{enumerate}
 \item [(a)] $R$ is a strong $(u,x,y)$-rope-bridge with ropes $\rho_x$ and $\rho_y$.
 \item [(b)] If, for $z\in \{x,y\}$, $C$ is a cycle of $R$, then $R$ has no $(u,z)$-path edge-disjoint from $C$.
\end{enumerate}
\end{lemma}
\begin{proof} It follows from Lemma \ref{str-char1} that (a) implies (b). Suppose (b). This implies item (b) of Lemma \ref{rope-bridge criterion}. So, $R$ is a $(u,x,y)$-rope-bridge with ropes $\rho_x$ and $\rho_y$. Suppose that $R$ is not strong. Then, some step $\sigma$ has an inner vertex $z$ or two steps $\alpha$ and $\beta$ have a common end-vertex. In the former case $\rho_x$ is edge-disjoint from the cycle $u,z,\sigma,y_\sigma,\rho_y,u$, a contradiction. In the later case, we may assume that $x_\alpha=x_\beta$. This implies that $\rho_x$ is edge-disjoint from the cycle $x_\beta,\beta,y_\beta,\rho_y,y_\alpha,\alpha,x_\alpha$, a contradiction again. So, (a) holds.
\end{proof}

\begin{proofof}\emph{Proof of Theorem \ref{main-strong}: } First note that all graphs described in the theorem are strongly $e$-Dirac. Let us prove the converse. Let $G$ be a strongly $e$-Dirac graph. If $G$ is a type (a) or (b) $e$-Dirac graph as in Theorem \ref{main}, then the result follows from Lemmas \ref{str type a} and \ref{str type b} respectively. Assume that $G$ is a type (c) $e$-Dirac graph. If $|G|\le 5$, the result may be verified directly. Thus we may assume that $|G|\ge 6$ and, by Lemma \ref{complete}, the assumptions for the previous results are valid. So, item (c) of Theorem \ref{main-cor} holds and we have the hypothesis of Lemma \ref{components are rope-bridges} holding, consider the terminologies as in that lemma. As $|G|\ge 6$ we may assume the existence of some $u_j$-component. 

Let $K_1,\dots, K_m$ be the $u_j$-components of $G$. For each $z\in\{x,y\}$ there is a $(z,u_i)$-path $\alpha_z$ with all internal vertices out of $K_1\cup\cdots\cup K_m\cup S^+$.

If item (b) of Lemma \ref{components are rope-bridges} holds then, $R:=K_1+\{u_j,x,y\}$ is an $(u_j,x,y)$-rope bridge. If for some cycle $C$ of $R$ and $z\in\{x,y\}$, there is an $(u_j,z)$-path $\beta$ edge-disjoint from $C$, then $C$ is edge disjoint from a cycle of the form $u_i,u_j,\beta,z,\alpha_z,u_i$, a contradiction. So, there are no such path and cycle. Thus, by Lemma \ref{strong rope-bridge criterion}, $R$ is a strong $(u_j,x,y)$-rope bridge. 

By Lemma \ref{str-char1}, this implies the theorem if all $\{i,j\}=\{1,2\}$ satisfies items (a) or (b) of Lemma \ref{components are rope-bridges}. So, we may assume that item (c) holds for $(i,j)=(1,2)$.

If for some $z\in\{x,y\}$, there is an $(u_2,z)$-path $\gamma$ whose internal vertices are in an $u_j$-component other than $K_1$ and $K_2$, then $u_1,u_2,\gamma,z,\alpha_z,u_1$ is edge-disjoint from a cycle of the form $x,K_1,y,K_2,x$, a contradiction. So, $m=2$ and $u_2x,u_2y\notin E(G)$.

If $u_2$ has two neighbors in a same $u_2$-component, say $K_2$, then $K_2+u_2$ has a cycle which is edge-disjoint from a cycle of the form $u_1,u_2,K_1,x,\alpha_x,u_1$; a contradiction. So, $u_2$ has an unique neighbor in each $u_2$-component. By item (c) of Lemma \ref{components are rope-bridges}, $|K_1|=|K_2|=1$.

If there is no $u_1$-component, either $G\cong K_{3,3}$  and the theorem holds or $xy\in E(G)$ and cycles of the form $x,y,K_1,x$ and $u_1,u_2,k_2,x,u_1$ are edge-disjoint, a contradiction. So, we may assume that $J$ is a $u_1$-component. If there is a second $u_1$-component $J'$, then  cycles of the form $u_1,u_2,K_1,x,J,u_1$ and $x,K_2,y,J',x$ are edge-disjoint. Thus, $J$ is the unique $u_1$-component. Hence, item (b) of Lemma \ref{components are rope-bridges} holds for $(i,j)=(2,1)$. If $R$ has a cycle $C$ and aj $(u_1,z)$-path $\gamma$ disjoint from $C$ for some $\in \{x,y\}$, then a cycle of the from $u_1,\gamma,z,K_1,u_2,u_1$ is edge-disjoint from $C$, a contradiction. By Lemma \ref{strong rope-bridge criterion}, $R$ is a strong $(u_1,x,y)$-rope bridge. If $xy\in E(G)$, then cycles of the form $u_1,u_2,K_1,x,J,u_1$ and $x,y,K_2,x$ are edges-disjoint. so $xy\notin E(G)$. Consider labels for the vertices of $R$ like in Lemma \ref{str-char1}. We let $V(K_1):=\{v_{n+1}\}$ and $V(K_2):=\{w_{n+1}\}$. Now it is straightforward to check that $G$ is in the format described in the theorem. \end{proofof}

\begin{proofof}\emph{Proof of Theorem \ref{main-strong-2con}: } Consider a graph $G$ as described in the theorem. Note that all cycles not contained in one of the $G_i$'s are the cycles containing $U$, which are exactly the cycles containing $e$. As each $G_i$ is $u_{i-1}u_i$-Dirac, it follows that $G$ is $e$-Dirac. For the converse, suppose for a contradiction that $G$ is a $2$-connected strongly $e$-Dirac graph not fitting into the description of the theorem. 

The $3$-connected strongly $e$-Dirac graphs, described in Theorem \ref{main-strong}, fit into the description in this theorem. So $G$ is not $3$-connected. As the theorem also holds if $|G|\le 3$, then $G$ has a $2$-vertex-cut $\{x,y\}$. This implies that we may write $G$ as the union of two graphs $H$ and $K$ such that $|H|,|K|\ge 3$, $V(H)\cap V(K)=\{x,y\}$, $H+xy$ and $K+xy$ are $2$-connected, $e\in V(H)$ and $xy\notin E(K)$. 

Let us check that $H+xy$ is a strongly $e$-Dirac graph. Suppose for a contradiction that $H+xy$ has a pair of edge-disjoint cycles $(D_1,D_2)$ with $e\in D_1$. Then for some $i\in \{1,2\}$, $D_i$ is not a cycle of $G$. So, $xy\in E(D_i)-E(H)$. Let $D$ be a cycle of $K+xy$ containing $xy$. Now $(D_i\cup D)\del xy$ and $D_{3-i}$ are disjoint cycles of $G$ whose union contains $e$, contradicting the fact that $G$ is strongly $e$-Dirac. So $H+xy$ is a $2$-connected strongly $e$-Dirac graph. As $|H|<|G|$, the theorem holds for $H+xy$. Consider, for $H+xy$, graphs $G'_1,\dots,G'_{n'}$ and vertices $u'_0,\dots,u'_{n'}$ as in the theorem, with $e=u'_0u'_{n'}$.

As $K+xy$ is $2$-connected and $xy\notin E(K)$, then either $K$ has a cycle or $K$ is an $xy$-path. In the later case, $G$ is isomorphic to a subdivision of $H$ and as, the theorem holds for $H$, it is straighforward to verify that it also holds for $G$, so $K$ has a cycle $C_K$.

If $H$ has a cycle $C_H$ containing $e$, then $C_H$ and $C_K$ are edge-disjoint cycles of $G$, a contradiction. So $e$ is in no cycle of $H$. This implies that $xy\notin E(H)$ and $\{e,xy\}$ is an edge-cut of $H+xy$ since $H+xy$ is $2$-connected. By the description of $H+xy$ as in the theorem, $\{x,y\}=\{u'_{i-1},u'_i\}$ for some index $i\in \{1,\dots,n'\}$ such that $V(G'_i)=\{u'_{i-1},u'_i\}$. We may assume without loss of generality that $(x,y)=(u'_{i-1},u'_i)$.

Let us check that $K+xy$ is strongly $xy$-Dirac. Suppose for a contradiction that $K+xy$ has a pair of edge-disjoint circuits $(C_1,C_2)$ with $xy\in E(C_1)$. Then for a circuit $C$ of $H+xy$ with $e,xy\in E(C)$, $((C\cup C_1)\del xy,C_2)$ is a pair of edge-disjoint cycles of $G$ whose union contains $e$, a contradiction. So $K+xy$ is a $2$-connected strongly $xy$-Dirac graph. As $|K|<|G|$, we may apply the theorem in $K+xy$ in respect to the edge $xy$. Consider, for $K+xy$, graphs $G''_1,\dots,G''_{n''}$ and vertices $u''_0,\dots,u''_{n''}$ as in the theorem with $(x,y)=(u''_0,u''_{n''})$.

Now the graphs 
\[(G_1,\dots,G_n):=(G'_1,\dots,G'_{i-1},G''_1,\dots,G''_{n''},G'_{i+1},\dots,G'_{n'})\] and vertices 
\[(u_0,\dots,u_n):=(u'_0,\dots,u'_{i-1},u''_1,\dots,u''_{n''-1},u'_{i},\dots,u'_{n'})\] give a description of $G$ according to the theorem.
\end{proofof}

\section{Prism-Minors}\label{sec-prism}

In this section we prove Theorem \ref{thm-prism}. The theorem  follows straightforwardly from Lemmas \ref{prism a}, \ref{prism b} and \ref{prism c}.

If $G$ is a graph with a subgraph $H'$ isomorphic to the subdivision of a graph $H$, we say that and $H$-minor of $H'$ is an \defin{$H$-topological minor} of $G$. If $G$ has an $H$-topological minor using an edge $e$, then it is clear that $G$ has an $H$-minor using $e$. The converse does not hold in general, but it is easy to verify that it is true provided $G$ and $H$ are $3$-connected and $H$ is cubic, which is the case in our concern: when $H$ is the prism and $G$ is $3$-connected. We will use this fact with no mentions. 

Let $G$ be a $3$-connected graph with an edge $e$. By Menger's Theorem, $G$ is not $e$-Dirac if and only if $G$ has a prism-minor $H$ using $e$ as an edge in a triangle of $H$. So, our problem lies within the class of $e$-Dirac graphs. Moreover, the following lemma is valid. 

\begin{lemma}\label{prism-menger}
A $3$-connected $e$-Dirac graph has a prism-minor using $e$ if and only if it has vertex-disjoint cycles $C$ and $D$ and three vertex-disjoint $(V(C),V(D))$-paths $\alpha_1$, $\alpha_2$, and $\alpha_3$ such that $e\in E(\alpha_3)$.
\end{lemma}

The next two lemmas proves Theorem \ref{thm-prism} for $e$-Dirac graphs of types (a) and (b).

\begin{lemma}\label{prism a}
If $G$ is a type (a) $e$-Dirac graph, then $G$ has a prism-minor using $e$.
\end{lemma}
\begin{proof}
For some $n\ge 4$, we may assume that $G$ may be obtained from $W_n$ by possibly doubling some spokes and splitting the hub into the edge $e=uv$, where $u$ and $v$ has degree at least four. Say that $|N_G(v)|\ge|N_G(u)|$. Choose $\{a,b\}\cont N_G(u)-v$ minimizing the number of neighbors of $v$ in $\{a,b\}$. 

If $n\ge 5$, by the minimality of $|\{a,b\}\cap N_G(v)|$ and as $|N_G(v)|\ge|N_G(u)|\ge 4$, there are at least three neighbors of $v$ out of $\{a,b,u\}$. So $v$ have neighbors $c$ and $d$ with the property that $a$, $b$, $c$ and $d$ appear in this order in some cycle ordering of the cycle $G-\{u,v\}$. This implies that $G$ has a subdivision of the prism containing $e$ and this implies the lemma.

So we may assume that $n=4$. Recall that, by the description of the type (a) graphs, $d_G(u)=d_G(v)\ge4$. Let $a$, $b$, $c$, and  $d$ be a cycle ordering of $G-\{u,v\}$ with the property that $b$, $c$, and $d$ are neighbors of $v$. Either $b$ or $d$ is a neighbor of $u$, we may assume it is $b$ as swapping the labels just inverts the cycle ordering. Now $\{u,a,b\}$ and $\{v,c,d\}$ induces triangles in $G$. But $ad$, $bc$ and $e=uv$ are edges of $G$. So $e$ is in a prism-minor of $G$. This finishes the proof.
\end{proof}

\begin{lemma}\label{prism b}
If $G$ is a type (b) $e$-Dirac graph, then $G$ has no prism-minor using $e$.
\end{lemma}
\begin{proof}
For $e=uv$ and some vertex $w$, $A:=\{u,v,w\}$ is a vertex-cut of $G$ and each connected component of $G-A$ is a tree with an unique neighbor of $w$. Suppose for a contradiction that there is a prism-minor of $G$ using $e$. By lemma \ref{prism-menger}, $G$ has vertex disjoint cycles $C$ and $D$ and vertex-disjoint $(V(C),V(D))$-paths $\alpha_1$, $\alpha_2$ and $\alpha_3$ with $e\in E(\alpha_3)$. 

As the components of $G-A$ are trees, both $C$ and $D$ meet $A$. Since $w$ has an unique neighbor in each connected component of $G-A$, then it is not possible that $V(C)\cap A=\{w\}$ or $V(D)\cap A=\{w\}$. So, we may assume that $u\in V(C)$ and $v\in V(D)$.

For $i=1,2$, let $c_i$ and $d_i$ be the endvertices of $\alpha_i$ in $C$ and $D$ respectively. Consider the cycle:
\[C'=c_1,C-u,c_2,\alpha_2,d_2,D-v,d_1,\alpha_1,c_1.\]
As the connected components of $G-A$ are trees, $C'$ meets $A$. But $u,v\notin V(C')$. Hence $V(C')\cap A=\{w\}$. So $V(C')-w$ is entirely contained in a connected component of $G-A$, which, therefore, contains the two neighbors of $w$ in $C'$, a contradiction. 
\end{proof}


\begin{lemma}\label{prism-creterion}
Suppose that $G$ is a $3$-connected graph, $uv$ is an edge of $G$ and $G-\{u,v\}$ is $2$-connected. Then $G$ has a prism-minor containing $uv$ if and only if $G$ has vertex-disjoint cycles $C$ and $D$ such that $\{u,v\}\cont V(C)\cup V(D)$.
\end{lemma}
\begin{proof}
Suppose that $G$ has such cycles $C$ and $D$. If both $u$ and $v$ are in one of these cycles, say $C$, then we may choose $C$ in such a way that $uv\in E(C)$; by applying Menger's Theorem on $G$ to obtain three vertex-disjoint $(V(C),V(D))$-paths, we get a prism-minor of $G$ using $e$. So, assume that $u\in C$ and $v\in D$. Now we apply Menger's Theorem on $G-\{u,v\}$ to obtain two $(V(C)-u,V(D)-v)$ vertex disjoint-paths that, together with $u,v$, are three vertex-disjoint $(V(C),V(D))$-paths. So, $G$ has a prism-minor using $uv$ in all cases.

Conversely, suppose that $G$ has a prism minor using $uv$. If $G$ is not $uv$-Dirac the result follows from Menger's Theorem. So, we may assume that $G$ is $uv$-Dirac and, by Lemma \ref{prism-menger}, $G$ has vertex-disjoint cycles $C$ and $D$ and vertex-disjoint paths $(V(C),V(D))$-paths $\alpha_1$, $\alpha_2$, and $\alpha_3$ such that $uv\in E(\alpha_3)$. Make the choice of these cycles and paths minimizing $|\alpha_3|$. If $\{u,v\}\cont V(C)\cup V(D)$, we have nothing to prove. So we may assume that $\alpha_3$ has an inner vertex in $\{u,v\}$. Let $x$ and $y$ be the endvertices of $\alpha_3$ in $C$ and $D$ respectively. Let $X:=V(C)\cup V(D)\cup V(\alpha_1)\cup V(\alpha_2)$. As $G$ is $3$-connected, there is a $(V(\alpha_3)-\{x,y\},X)$ path $\beta$ in $G-\{x,y\}$. Let $a$ and $b$ be the endvertices of $\beta$, with $a\in V(\alpha_3)$. We may assume that the edge $uv$ is in the path $a,\alpha,y$. If $b\in V(D)\cup \int(\alpha_i)$ for some $i\in\{1,2\}$, then $D\cup\beta\cup(a,\alpha_3,y)\cup(a,\alpha_i,y)$ has a cycle containing $e$ and avoiding $V(C)$, a contradiction. So $b\in V(C)$. Let $c_1$ and $c_2$ be the respective endvertices of $\alpha_1$ and $\alpha_2$ in $C$. Now $C\cup\beta\cup(x,\alpha_3,a)$ has the cycle $C':=a,\alpha_3,x,C-c_1,b,\beta,a$. For some $\{i,j\}=\{1,2\}$, the cycles $C'$ and  $D$ and the vertex-disjoint $(V(C'),V(D))$-paths $(x,C-c_j,c_i,\alpha_i)$, $(b,C-c_i,c_j,\alpha_j)$ and $a,\alpha_3,y$ have the property that $uv\in E(b,\alpha_3,y)$ and $|a,\alpha_3,y|<|\alpha_3|$, contradicting the minimality of $|\alpha_3|$.
\end{proof}

\begin{lemma}\label{prism c}
If $G$ is a type (c) $e$-Dirac graph with at least six vertices, then $G$ has a prism-minor using $e$ or $G\cong W_n$, $K_{3,n}$, $K'_{3,n}$, $K''_{3,n}$ or $K'''_{3,n}$ for some $n\ge 3$.
\end{lemma}
\begin{proof}\setcounter{rot}{0}

Suppose that the lemma fails for $G$ and let $e=u_1u_2$. By Theorem \ref{main-cor}, there are vertices $x$ and $y$ such that $N_G(u_1)=\{u_2,x,y\}$, $N_G(u_2)=\{u_1,x,y\}$, or $\{u_1,u_2,x,y\}$ is a two-sided vertex cut of $G$. Let us use the terminology of that theorem. First we check:

\begin{rot}\label{prism c-0}
$d_G(u_1)=3$ or $d_G(u_2)=3$.
\end{rot}
\begin{rotproof}
Suppose the contrary. For $i=1,2$, let $K^i_1,\dots,K^i_{n_i}$ be the $u_i$-components. As $d_G(u_1),d_g(u_2)\ge 4$, it follows that $n_1,n_2\ge 1$. If, for some $\{z_1,z_2\}=\{x,y\}$ and for each $i=1,2$, there is a cycle $C_i$ of $(K^i_1\cup\cdots\cup K^i_{n_i})+\{z_i,u_i\}$ containing $u_i$, then $C_1$ and $C_2$ contradict Lemma \ref{prism-creterion}. So, we may assume that $H:=(K^1_1\cup\cdots\cup K^1_{n_1})+\{y,u_1\}$ has no cycle containing $u_1$.

If $n_1\ge 2$, then $H$ has a cycle of the form $u_1,K^1_1,y,K^1_2,u_1$, a contradiction. So $n_1=1$.

If $u_1y\in E(G)$, $H$ has a cycle of the form $u_1,K^1_1,y,u_1$, a contradiction again. So $u_1y\notin E(G)$.

As $d_G(u_1)\ge 4$, there are two edges from $u_1$ to $K^1_1$. So $K^1_1+u_1$ has a cycle containing $u_1$, which is a cycle of $H$, another contradiction.
\end{rotproof}

By \ref{prism c-0}, we may assume that $N_G(u_1)=\{u_2,x,y\}$. 

\begin{rot}\label{prism c-3}
There is exactly one $u_2$-component.
\end{rot}
\begin{rotproof}
Let $J_1,\dots, J_n$ be the distinct $u_2$-components with $n\ge 2$. 

Let us prove first that each one fo these components have exatly one vertex. Say that $J_1$ has more than one vertex. By the description of Theorem \ref{main-cor}, $J_1+\{u_1,x_1,x_2\}$ is either a wheel with $u_2$ as hub or a fan with $u_2$ as hub and $x$ and $y$ as endvertices. In each of these cases, $J_1+u_2$ has a cycle $C$ containing $u_2$. This cycle, together with a cycle of the form $u_1,x,J_2,y,u_1$ contradict Lemma \ref{prism-creterion}. So each $u_2$-component has an unique vertex.

This implies that each vertex of $J_1\cup\cdots\cup J_n\cup\{u_1\}$ has $\{x,y,u_2\}$ as its neighborhood. So, $G\cong K_{3,n+1}$, $K'_{3,n+1}$, $K''_{3,n+1}$ or $K'''_{3,n+1}$, implying the lemma, a contradiction.
\end{rotproof}

Now, by the description of Theorem \ref{main-cor}, $N_G(u_1)=\{x,y,u_2\}$ and $G-u_1$ is an $(u_2,x,y)$-rope bridge. We will denote by $\rho_x$ and $\rho_y$ its ropes. Instead of the ropes, we will argue using the paths $\pi_x:=\rho_x,x,u_1$ and $\pi_y:=\rho_y,y,u_1$ because they have a certain symmetry regarding $u_2$ and $u_1$. Let $\pi_x=u_2,x_1,\dots,x_{m_x},u_2$ and $\pi_y:=u_1,y_1,\dots,y_{m_y},u_1$. For a step $\sigma$ and $k\in\{1,2\}$, we denote by $C_k(\sigma)$ the cycle $u_k,\pi_x,x_\sigma,\sigma,y_\sigma,u_k$. 

\begin{rot}\label{prism c-5}
Each pair of steps either cross or have a common endvertex.
\end{rot}
\begin{rotproof}
If the claim fails, there are steps $\alpha$ and $\beta$ such that $u_1$, $z_\alpha$, $z_\beta$ and $u_2$ appear in this order in $\pi_z$ for each $z\in \{x,y\}$. Using Lemma \ref{prism-creterion} for the cycles $C_1(\alpha)$ and $C_2(\beta)$ we conclude that $G$ has a prism minor using $e$, a contradiction.
\end{rotproof}

\begin{rot}\label{prism c-6}
Each step has at most one inner vertex.
\end{rot}
\begin{rotproof}
Suppose that the claim fails and suppose that $\alpha:=v_1,\dots,v_n$ is a step with $n\ge 4$, $v_1\in V(\pi_x)$ and $v_n\in V(\pi_y)$. If $\alpha$ is the unique step, then $G$ is a wheel with $u_2$ as hub and the lemma holds. So, there is another step $\beta$. Now the cycles $u_2,v_1,v_3,u_1$ and $C_1(\beta)$ yield the existence of a prism-minor of $G$ using $e$ by  Lemma \ref{prism-creterion}, a contradiction.
\end{rotproof}

\begin{rot}\label{prism c-7}
Let $z\in\{x,y\}$ and let $z_k$ be an inner vertex of $\pi_z$. Suppose that $uz_k\in E(G)-E(\pi_z)$ or a step with an inner vertex contains $z_k$. Then no step has an endvertex in $\{z_{k+1},\dots,z_{m_z}\}$.
\end{rot}
\begin{rotproof}
Suppose that the claim fails. Say that $z=x$. So, there is a step $\beta$ arriving in $x_l$ for some $k<l\le m_x$. If $ux_k\in E(G)-E(\pi_x)$, the cycles $u,\pi_x,z_k,u$ and $C_1(\beta)$ contradict Lemma \ref{prism-creterion} since $G$ has no prism-minor using $e$. So assume that a step $\alpha$ with an inner vertex $w$ has $z_k$ as endvertex. Now $u_2,\pi_x,z_k,\alpha,w,u_2$ and $C_1(\beta)$ contradict Lemma \ref{prism-creterion}.
\end{rotproof}

\begin{rot}\label{prism c-8}
Each pair of steps have a common vertex.
\end{rot}
\begin{rotproof}
By \ref{prism c-6}, it suffices to prove that no pair of steps cross. Suppose that $\alpha$ and $\beta$ are crossing steps. It follows from \ref{prism c-7} that these steps have no inner vertices. So, we may assume that there are indices $1\le a<c \le m_x$ and $1\le b\le d\le m_y$ such that $\alpha=x_a,y_c$ and $\beta=y_b,x_d$. 

Let us check that $c=a+1$. Suppose for a contradiction that there is an index $a<k<c$. As each step crosses at most one other step by (RB3), then no step has endvertex in $x_k$. Moreover, by \ref{prism c-7}, $ux_k\notin E(G)$. This implies that the degree of $x_k$ is two, a contradiction. So, $c=a+1$ and, analogously, $d=b+1$.

Let us check now that $a=1$. Suppose that $a>1$. As $d_G(x_1)\ge 3$, there is a step $\gamma$ with $x_1$ as endvertex. As each step crosses at most one other step, $\gamma$ does not cross $\alpha$ nor $\beta$. So, $\gamma$ intersects both $\alpha$ and $\beta$ by \ref{prism c-5}. But, for this to happen, it is necessary that $\gamma$ share a common endvertex in $\pi_y$ with both $\alpha$ and $\beta$, a contradiction. So $a=1$. Analogously, $b=1$.

Let us check that $m_x=2$. Suppose that $m_x\ge 3$. By (RB6) $ux_3\notin E(G)$. So, there is a step $\gamma$ arriving at $x_3$. As argued in the previous paragraph, $\gamma$ share a common vertex in $\pi_y$ with both $\alpha$ and $\beta$, a contradiction. Therefore, $m_x=2$ and, analogously, $m_y=2$.

If there are no other steps than $\alpha$ and $\beta$, then $G\cong K_{3,3}$, so there are other steps. As no other steps cross $\alpha$ nor $\beta$, each other step has endvertices in $\{x_1,y_1\}$ or $\{x_2,y_2\}$.

First suppose that there is a step $\gamma$ with endvertices in $\{x_1,y_1\}$. By \ref{prism c-7}, $\gamma$ may not have inner steps and this establishes the uniqueness of $\gamma$. By \ref{prism c-5}, there is no steps with endevertices in $\{x_2,y_2\}$. So $\alpha$, $\beta$ and $\gamma$ are the unique steps. By (RB6), $u_2x_2,u_2y_2\notin E(G)$ and, therefore, $G\cong K'_{3,3}$, a contradiction.

So, all steps differing from $\alpha$ and $\beta$ have endvertices in $\{x_2,y_2\}$. Note that $N_G(u_1)=N_G(x_1)=N_G(y_1)=\{u_2,x_2,y_2\}$. If $w\in V(G)-\{u_1,u_2,x_1,x_2,y_1,y_2\}$, then $w$ is an inner vertex of a step with endvertices in $\{x_2,y_2\}$. By \ref{prism c-6}, $N_G(w)=\{u_2,x_2,y_2\}$. this implies that each vertex out of $\{u_2,x_2,y_2\}$ have this set as neighborhood and $G\cong K_{3,n}$, $K'_{3,n}$, $K''_{3,n}$ or $K'''_{3,n}$ for some $n\ge 3$; but this implies the lemma.
\end{rotproof}

\begin{rot}\label{prism c-9}
$x_1,y_1$ is the unique step with $x_1$ and $y_1$ as endvertices, but not the unique step.
\end{rot}
\begin{rotproof}
First we prove that there is a step $\beta$ with $x_1$ and $y_1$ as endvertices. Indeed, as $d_G(x_1), d_G(y_1)\ge 3$, there must be a step $\alpha$ containing $x_1$ and a step $\beta$ containing $y_1$. We may assume that $\alpha$ contains a vertex $y_k$ with $k>1$. If the endvertex of $\beta$ in $\pi_x$ is not $x_1$, then $\alpha$ and $\beta$ cross, contradicting \ref{prism c-8}. So there is a step $\beta$ with $x_1$ and $y_1$ as endvertices. 

Now let us prove the uniqueness of $\beta$. Suppose that $\gamma$ is a second step with endvertices in $x_1$ and $y_1$. As $G$ has no parallel edges, one of $\gamma$ or $\beta$ has an inner vertex. By \ref{prism c-7} no step meets $Z:=\{x_2,\dots,x_m,y_2,\dots,y_n\}$. By (RB6) there is no edge joining $u_2$ and a vertex of $Z$. Thus the vertices of $Z$ have degree two and $Z$ must be empty. This implies that each vertex $z \in V(G)-\{u_1,u_2,x_1,y_1\}$ is an inner vertex of some step with $x_1$ and $y_1$ as endvertices; by \ref{prism c-6} $N_G(z)=\{u_2,x_1,y_1\}$ and, as $N_G(u_1)=\{u_2,x_1,y_1\}$, this implies that $G\cong K_{3,3}$, $K'_{3,3}$, $K''_{3,3}$ or $K'''_{3,3}$, a contradiction. Thus the uniqueness of $\beta$ as a step with endvertices in $\{x_1,y_1\}$ is established.

Let us prove that $\beta$ is not the unique step. Assume the contrary. This implies that no step has an endvertex in the set $Z$ defined in the last paragraph. Thus $u_2z\in E(G)$ for all $z\in Z$. This implies that $N_G(u_2)=V(G)-u_2$. But the unique edges not incident to $u_2$ are those in $E(\pi_x)-u_2x_1$, $E(\pi_y)-u_2y_1$ or $E(\beta)$. But $u_1,\pi_x,x_1,\beta,y_1,\pi_y,u_1$ induces a cycle in $G$, and, therefore, $G$ is a wheel with $u_2$ as hub, a contradiction.

As there is a second step $\gamma$, by the uniqueness of $\beta$, $\gamma$ arrives at a vertex of $Z$, then by \ref{prism c-7}, $\beta$ has no inner vertex and the claim holds.
\end{rotproof}

Now we finish the proof. By \ref{prism c-9}, there is a step $\alpha\neq x_1,y_1$. By \ref{prism c-8} and by the uniqueness of $x_1y_1$ established in \ref{prism c-9}, we may assume that the endvertices of $\alpha$ are $y_1$ and $x_k$ for some $2\le k\le m_x$. Choose $\alpha$ with $k$ as small as possible. 

By \ref{prism c-8}, each step must intersect $\alpha$ and $x_1,y_1$. As $x_1\notin V(\alpha)$, each steps contains $y_1$. Therefore, for  $2\le l \le n$, there is no step arriving in $y_l$. Moreover, by (RB6), $u_2y_l\notin E(G)$. So $d_g(y_l)=2$ and such an index $l$ may not exist. So, $m_y=1$.

Let us check that $k=2$. Suppose for a contradiction that there is an index $1<l<k$. By \ref{prism c-7}, $u_2x_l\notin E(G)$ and, therefore there is a step $\beta$ arriving at $x_l$. But $y_1$ must be an endvertex of $\beta$, thus $\beta$ and $l$ contradict the minimality of $\alpha$ and $k$. So, $k=2$

As $|V(G)|\ge 6$, $m_x\ge 3$. Now we check, for each $l=1,\dots m_x$, that $u_2x_l\notin E(G)$ and that each step arriving at $x_l$ has no inner vertices. Indeed, for $l\ge 3$, this follows from (RB6). For $l\le 2$, this follows from \ref{prism c-7}. In particular this implies that  $x_ly_1$ is the unique edge of $G$ out of $\pi_x$ incident to $x_l$. Now $G$ is a wheel with $y_1$ as hub. This proves the lemma.
\end{proof}

Theorem \ref{thm-prism} now follows from Lemmas \ref{prism a}, \ref{prism b} and \ref{prism c}.


\end{document}